\documentclass[11pt, oneside]{article}   	
\usepackage{geometry}                		
\usepackage{graphicx}				
\usepackage{amssymb}
\usepackage{amsmath}
\usepackage{amsthm}
\usepackage{mathrsfs}
\title{Decay rates for the damped wave equation with finite regularity damping}
\author{Perry Kleinhenz}
\date{}							

\theoremstyle{definition}
\newtheorem{definition}{Definition}

\theoremstyle{theorem}
\newtheorem{theorem}{Theorem}[section]
\newtheorem{lemma}[theorem]{Lemma}

\newtheorem{proposition}[theorem]{Proposition}
\newcommand{\Bc}{\mathcal{B}}

\newcommand{\Rb}{\mathbb{R}}
\newcommand{\Zb}{\mathbb{Z}}
\newcommand{\Rn}{\mathbb{R}^n}

\newcommand{\N}{\mathbb{N}}

\newcommand{\Dc}{\mathcal{D}}

\newcommand{\F}{\mathcal{F}}

\newcommand{\Nb}{\mathbb{N}}

\newcommand{\intr}{\int_{\Rb}}

\newcommand{\ra}{\rightarrow}

\newcommand{\<}{\left\langle}
\renewcommand{\>}{\right\rangle}
\newcommand{\e}{\varepsilon}
\renewcommand{\d}{\delta}

\newcommand{\nm}[1]{\left| \left| #1 \right| \right|}

\newcommand{\lp}[2]{ \nm{#1}_{L^{#2}}}

\newcommand{\hp}[2]{\nm{#1}_{H^{#2}}}

\newcommand{\ltwo}[1]{\lp{#1}{2}}

\newcommand{\Sw}[1]{\nm{#1}_{S_w}}
\newcommand{\Del}{\Delta}

\newcommand{\Cs}{C^{\infty}_0}

\newcommand{\p}{\partial}
\newcommand{\Sc}{\mathcal{S}}

\newcommand{\Ci}{C^{\infty}}

\newcommand{\supp}{\text{supp }}
\newcommand{\T}{\mathbb{T}}
\newcommand{\Ac}{\mathcal{A}}

\newcommand{\Ls}{\mathscr{L}}

\newcommand{\ti}{\widetilde}

\newcommand{\Smr}{S^m_{\rho}(\Ps) }
\newcommand{\Smrp}{S^{m'}_{\rho}(\Ps)}
\newcommand{\Smrpp}{S^{m+m'}_{\rho}(\Ps)}

\newcommand{\lpfu}{\lp{f}{2}\lp{u}{2}}
\newcommand{\Olt}{O_{L^2\ra L^2}}
\newcommand{\olt}{o_{L^2 \ra L^2}}

\newcommand{\Op}{\text{Op}}

\newcommand{\Sb}{\mathbb{S}}

\newcommand{\Ps}{T^* \mathbb{S}^1}

\newcommand{\wtbk}{{b_{k}}}
\newcommand{\bkl}{b_{k,l}}
\newcommand{\bklt}{\ti{b_{k,l}}}
\newcommand{\tib}{\ti{b}}
\newcommand{\tm}{\tau_{min}}

\newcommand{\bigbadestimatefirstterm}{\frac{h^{-\tau}}{\beta}}
\newcommand{\Wkinf}{W^{k_0,\infty}}

\newcommand{\dampingmaincommutatorlemmaeq}{h^{3-\gamma} \left| \<\Op(\chi \psi W) u, u\>\right| \leq C h^{3}\lpfu+ C h^{2+7\tau-\gamma} \lp{f}{2}^2 + o(h^3) \lp{u}{2}^2.}

\newcommand{\awkwardtermmaincommutatorlemmaeq}{2h \left| \<\Op(\xi^2 \psi(\xi h^{\tau-1}) x \chi'(x))u ,u\> \right|\leq Ch^{3-2\tau+\gamma}\lpfu + O(h^{\infty}) \lp{u}{2}^2.}

\newcommand{\chipsiPintofeq}{h |\<\Op(\chi\psi p)u, u\>| \leq Ch^3 \lpfu + C h^{2+6\tau-\gamma}\lp{f}{2}^2 +o(h^3) \lp{u}{2}^2.}

\newcommand{\boundbelowJeq}{h^3 \beta |\<\Op(\chi \psi)u, u\>| \geq h^3 \beta \lp{Ju}{2}^2 - C \beta h^{3+2\tau} \lp{u}{2}^2.}

\newcommand{\doublecommutatoreq}{[[h^{1-\tau}A, v_j], v_j] = \sum_{l=1}^{k_0-6} \sum_{k=1}^{k_0-6} \frac{i^{m+l}}{2^{k+l} k! l!} h^{\tau(k+l)+1-\tau} (1-(-1)^k)(1-(-1)^l) \Op(\p_x^k v_j \p_x^l v_j a^{(k+l)}) + o(h^3)}

\begin{document}
\maketitle

\begin{abstract}
Decay rates for the energy of solutions of the damped wave equation on the torus are studied. In particular, damping invariant in one direction and equal to a sum of squares of nonnegative functions with a particular number of derivatives of regularity is considered. For such damping energy decays at rate $1/t^{2/3}$. If additional regularity is assumed the decay rate improves. When such a damping is smooth the energy decays at $1/t^{4/5-\d}$. The proof uses a positive commutator argument and relies on a pseudodifferential calculus for low regularity symbols. 
\end{abstract}

\section{Introduction}
Let $W$ be a bounded, nonnegative damping function on a compact Riemannian manifold $M$, and let $v$ solve
\begin{equation*} \label{DWE}
\begin{cases}
\p_t^2  v - \Del v + W(x) \p_t v = 0 & t>0,  \\
(v,\p_t v) |_{t=0} = (v_0, v_1) \in \Ci(M) \times \Ci(M) & t=0. \\
\end{cases}
\end{equation*}
The primary object of study in this paper is the energy 
$$
E(v,t) = \frac{1}{2} \int |\nabla v|^2 +  |\p_t v|^2  dx.
$$
When $W$ is continuous it is classical that uniform stabilization is equivalent to geometric control by the positive set of the damping. That is $E(t) \leq C r(t) E(0)$, with $r(t) \ra 0$ as $t \ra \infty$, if and only if there exists $L$, such that all geodesics of length at least $L$ intersect $\{W>0\}$. Furthermore, in this case the optimal $r(t)$ is exponentially decaying in $t$. 

When the geometric control condition does not hold decay is instead of the form. 
\begin{equation}\label{e:endecgen}
E(t)^{1/2} \leq C r(t) \left(\hp{v_0}{2} + \hp{v_1}{1}  \right).
\end{equation}
Then the optimal $r(t)$ depends on the geometry of $M$ and $\{W>0\}$, as well as properties of $W$ in a neighborhood of $\{W=0\}$. This paper explores this dependence for translation invariant damping functions on the torus, and proves decay of the form  
\begin{equation}\label{e:endec}
E(t)^{1/2} \leq C (1+t)^{-\alpha} \left(\hp{v_0}{2} +  \hp{v_1}{1} \right).
\end{equation}
Such decay is guaranteed on the torus with $\alpha=1/2$ when $\{W>0\}$ is open and nonempty by \cite{AnantharamanLeautaud2014}.

First, when the damping is a sum of squares of sufficiently regular $y$-invariant functions there is an improved decay rate. 
\begin{theorem}\label{simpletheorem}
Let $M$ be the torus $(\Rb/2\pi \Zb)_x \times (\Rb/2\pi \Zb)_y$. Suppose $W(x,y)=W(x)$ and satisfies 
\begin{enumerate}
	\item For some  $\sigma \in (0,\pi), W$ is bounded below by a positive constant for $x \in[-\pi, \pi] \backslash [-\sigma, \sigma]$,
	\item There exists $\sigma_1 \in (0, \pi-\sigma)$ and there exist functions $v_j(x) \geq 0, v_j \in W^{9,\infty}(-\sigma-\sigma_1, \sigma+\sigma_1)$, such that $W(x)=\sum_j v_j(x)^2$ on $(-\sigma-\sigma_1, \sigma+\sigma_1)$.
\end{enumerate}
Then there exists $C$ such that \eqref{e:endec} holds with $\alpha = \frac{2}{3}$. 
\end{theorem}
If the damping is instead smooth and $y$-invariant there is an additional improvement. 
\begin{theorem}\label{smooththeorem}
Let $M$ be the torus $(\Rb/2\pi \Zb)_x \times (\Rb/2\pi \Zb)_y$. Suppose $W(x,y)=W(x)$ and satisfies 
\begin{enumerate}
	\item For some  $\sigma \in (0,\pi), W$ is bounded below by a positive constant for $x \in[-\pi, \pi] \backslash [-\sigma, \sigma]$,
	\item $W \in \Ci(\Rb/2 \pi \Zb)$.
\end{enumerate}
Then for all $\e>0$ there exists $C$ such that \eqref{e:endec} holds with $\alpha = \frac{4}{5}-\e$. 
\end{theorem}
Both of these theorems are actually consequences of the following result. When the damping is a sum of squares of functions with $k_0$ derivatives there is an improved decay rate which depends on $k_0$.  
\begin{theorem}\label{maintheorem}
Let $M$ be the torus $(\Rb/2\pi \Zb)_x \times (\Rb/2\pi \Zb)_y$. Suppose $W(x,y)=W(x)$ and satisfies 
\begin{enumerate}
	\item\label{zerosethyp} For some  $\sigma \in (0,\pi), W$ is bounded below by a positive constant for $x \in[-\pi, \pi] \backslash [-\sigma, \sigma]$,
	\item\label{squaredecomphyp} There exists $k_0 \geq 9, \sigma_1 \in (0, \pi-\sigma)$ and there exist functions $v_j(x) \geq 0, v_j \in \Wkinf(-\sigma-\sigma_1, \sigma+\sigma_1)$, such that $W(x)=\sum_j v_j(x)^2$ on $(-\sigma-\sigma_1, \sigma+\sigma_1)$.
\end{enumerate}
Let $\tm > \max \left(\frac{k_0+2}{2k_0-4}, \frac{7}{k_0-1} \right)$ then there exists $C$ such that \eqref{e:endec} holds with $\alpha = \frac{2}{\tm+2}$. 
\end{theorem}
\textbf{Remarks} \begin{itemize}
	\item The two constraints for $\tm$ in terms of the regularity $k_0$ are needed to guarantee error terms in composition expansions are small. In particular $\tm>\frac{k_0+2}{2k_0-4}$ is needed to ensure \eqref{eqellipticremainder} holds and $\tm>\frac{7}{k_0-1}$ is needed to ensure \eqref{sharptm2} holds. These constraints are sharp on these inequalities, but are also used in other estimates in the proof. 
	\item Theorem \ref{simpletheorem} is just Theorem \ref{maintheorem} when $k_0=9$. So $\tm$ can be taken $=1$ which gives decay at $\alpha=2/3$. 
	\item On the other hand by \cite{Bony2005} if $W \in C^{2k_0}(a,b)$ then there exist $v_1, v_2 \in C^{k_0}(a,b)$ such that $W=v_1^2 +v_2^2$ on $(a,b)$. Therefore if $W \in C^{2k_0}(-\sigma-\sigma_1, \sigma+\sigma_1)$ it satisfies hypothesis \ref{squaredecomphyp} of the theorem. Theorem \ref{smooththeorem} then follows from Theorem \ref{maintheorem} and the result of Bony. In particular for any fixed $k_0$ there is an appropriate expansion and so $\tm$ can be taken arbitrarily close to $1/2$ which gives decay at $\alpha = 4/5-\delta$. 
\end{itemize}

The equivalence of uniform stabilization and geometric control for continuous damping functions was proved by Ralston \cite{Ralston1969}, and Rauch and Taylor \cite{RauchTaylor1975} (see also \cite{BardosLebeauRauch1992} and \cite{BurqGerard1997}, where $M$ is also allowed to have a boundary). For some more recent finer results concerning discontinuous damping functions, see Burq and G\'erard \cite{BurqGerard2018}.

Decay rates of the form \eqref{e:endecgen} go back to Lebeau \cite{Lebeau1996}. When $W \in C(M)$ is nonnegative and $\{W>0\}$ is open and nonempty, then decay of the form \eqref{e:endecgen} holds with $r(t)=1/\log(2+t)$ in  \cite{Burq1998, Lebeau1996}. Furthermore, this is optimal on spheres and some other surfaces of revolution \cite{Lebeau1996}.  At the other extreme, if $M$ is a negatively curved (or Anosov) surface, and $W \in C^\infty(M)$, $W$ nonnegative and not identically zero, then \eqref{e:endecgen} holds with $r(t)=Ce^{-ct}$ \cite{DyatlovJinNonnenmacher}. 

When $M$ is a torus, these extremes are avoided and the best bounds are polynomially decaying as in \eqref{e:endec}. Anantharaman and L\'eautaud \cite{AnantharamanLeautaud2014} show \eqref{e:endec} holds with $\alpha =1/2$ when $W \in L^\infty$, $W \ge 0$, and $W>0$ on some open set, as a consequence of Schr\"odinger observability/control  \cite{Jaffard1990, Macia2010, BurqZworski2012}. The more recent result of Burq and Zworski on Schr\"odinger observability and control \cite{BurqZworski2019} weakens the final requirement  to merely $W \not \equiv 0$. Anantharaman and L\'eautaud \cite{AnantharamanLeautaud2014} further show that if $\supp W$ does not satisfy the geometric control condition then \eqref{e:endec} cannot hold for any $\alpha>1$. They also show if there exists $C>0$ such that $W$ satisfies $|\nabla W| \leq C W^{1-\e}$ for $\e<1/29$ and $W \in \Wkinf$ for $k_0 \geq 8$ then $\eqref{e:endec}$ holds with $\alpha =1/(1+4\e)$. 

Note that Theorem \ref{maintheorem} improves the dependence between $|\nabla W| \leq W^{1-\e}$ estimates and decay rate with slightly different hypotheses. That is a damping satisfying the hypotheses of Theorem \ref{maintheorem} has $|\nabla W| \leq C W^{1/2}$, which, if the \cite{AnantharamanLeautaud2014} result applied to $\e=1/2$, would only give \eqref{e:endec} with $\alpha =1/3$, no better than the generic upper bound, whereas Theorem \ref{maintheorem} gives \eqref{e:endec} with at least $\alpha =2/3$. 

Additionally, because of the result in \cite{Bony2005}, Theorem \ref{maintheorem} applies to sufficiently regular damping, which is invariant in one direction, without additional hypotheses. In particular \cite{AnantharamanLeautaud2014} mention that their results do not give an improvement over the Schr\"odinger observability bound for smooth damping vanishing like $W=e^{-1/x} \sin(1/x)^2$, while Theorem \ref{maintheorem} does. 

For earlier work on the square and partially rectangular domains see \cite{LiuRao2005} and  \cite{BurqHitrik2007} respectively, and for  polynomial decay rates in the setting of a degenerately hyperbolic undamped set, see \cite{csvw}. 

In \cite{Kleinhenz2019}, it was shown that if $W=(|x|-\sigma)_+^{\beta}$ near $[-\sigma,\sigma]$, then \eqref{e:endec} holds with  $\alpha=(\beta+2)/(\beta+4)$ and cannot hold for all solutions with $\alpha>(\beta+2)/(\beta+3)$. In the case of constant damping on a strip the result that \eqref{e:endec} holds with  $\alpha = 2/3$ is due to Stahn \cite{Stahn2017}, and the result that it does not hold for $\alpha>2/3$ is due to Nonnenmacher \cite{AnantharamanLeautaud2014}. In \cite{DatchevKleinhenz2019} it was shown that for $W \sim (|x|-\sigma)_+^{\beta}$ near $[-\sigma,\sigma],$ \eqref{e:endec} holds with $\alpha=(\beta+2)/(\beta+3)$, which is sharp when $W=(|x|-\sigma)_+^{\beta}$ near $[-\sigma,\sigma]$. 

These results along with Theorem \ref{maintheorem} suggest that sharp decay rate on the torus could be determined by the regularity of the damping at the boundary of its support. The other likely alternative is that the sharp decay rate is determined by the value of $\e$ for which $W$ satisfies $|\nabla W| \leq C W^{1-\e}$. Although the sharp decay rate for polynomial damping $W=(|x|-\sigma)_+^{\beta}$ depends on $\beta$, this does not disambiguate between these cases as $W \in W^{\beta, \infty}$ and $W$ satisfies $|\nabla W| \leq W^{1-1/\beta}$. A good candidate for distinguishing these is $W$ smooth and vanishing like $e^{-1/x} \sin(1/x)^2$, as it only satisfies $|\nabla W| \leq C W^{1/2}$. 

If regularity determines the sharp decay rate for any $\d>0$ such an oscillating damping should decay at $1/t^{1-\d}$ as there are other smooth dampings which decay this fast. As in \cite{AnantharamanLeautaud2014}, a smooth damping vanishing like $e^{-1/x}$ satisfies $|\nabla W| \leq C W^{1-\e}$ for any $\e>0$ and so for any $\d>0$ decays at $1/t^{1-\d}$. If on the other hand the derivative bound condition $|\nabla W| \leq C W^{1-\e}$ determines the sharp decay rate, the fact that $W=(|x|-\sigma)_+^2$ also satisfies $|\nabla W| \leq W^{1/2}$ and has solutions which decay no faster than $1/t^{4/5}$, means an oscillating damping also should have solutions which decay no faster than $1/t^{4/5}$. Theorem \ref{maintheorem} does not guarantee or rule out either of these, so resolving this question would be an interesting area for future work. 


\textbf{Acknowledgements} 
I would like to thank Jared Wunsch and Oran Gannot for helpful conversations and comments on early drafts. I would also like to thank the anonymous referee for their helpful comments, which improved the clarity of the paper and led to an improvement in the overall result. I would also like to thank Andras Vasy for helpful conversations while I was improving the result. I was partially supported by the National Science Foundation grant RTG: Analysis on Manifolds at Northwestern University. 

\subsection{Outline of Proof}
By a Fourier transform in time, it is enough to study the associated stationary problem. More precisely, by Theorem 2.4 of \cite{BorichevTomilov2010}, as formulated in Proposition 2.4 of \cite{AnantharamanLeautaud2014}, decay with $\alpha=\frac{2}{\tm+2}$ follows from showing that there are constants $C, q_0 >0$ such that, for any $q\geq q_0$, 
\begin{equation*}
\nm{(-\Del +iqW -q^2)^{-1}}_{L^2(\T^2) \ra L^2(\T^2)} \leq C q^{1/\alpha-1}=Cq^{\tm/2}.
\end{equation*}
Because the damping $W$ depends only on $x$ this can be reduced to a 1 dimensional problem by expanding in a Fourier series in the $y$ variable. Let $k$ be the vertical Fourier mode, set $\beta=q^2-k^2,$ take $f \in L^2(\Rb/2\pi \Zb)$ and consider $u \in H^2(\Rb/2\pi \Zb)$ solving 
\begin{equation}\label{1dDWE}
-u''+iqWu - \beta u = f.
\end{equation}
Then it is enough to show that there are $C, q_0>0,$ such that for any $f,$ any $q \geq q_0$ and any real $\beta \leq q^2$,  if $u$ solves \eqref{1dDWE} then
\begin{equation*}
\int |u|^2 \leq C q^{\tm} \int |f|^2. 
\end{equation*}
Here, and below, all integrals are over $\Rb/ 2\pi \Zb$. A more precise dependence on $\beta$ is obtained, for any $\e_1, \e_2>0$ there exists a constant $C$ such that 
\begin{equation}\label{lowenergy}
\int |u|^2 \leq C \int |f|^2, \quad \text{ when } \beta < \frac{\pi^2}{16(\sigma+\e_1)^2}, \quad q \geq q_0,
\end{equation}
and 
\begin{equation}\label{highenergy}
\int |u|^2 \leq C q^{\tm} \int|f|^2, \quad \text{ when } \e_2 < \beta \leq q^2, \quad q \geq q_0.
\end{equation}
It is clear that for $\e_1,\e_2$ small enough, \eqref{lowenergy} and \eqref{highenergy} cover all $\beta \leq q^2$. This $\beta$ can be thought of as the ``horizontal energy" of the solution. The larger it is, the larger $u$ is relative to $q$ in the $\xi$ direction in phase space. 

Equation \eqref{lowenergy} is the low horizontal energy case and is proved in section \ref{lowenergysection}. Equation \eqref{highenergy} is the high horizontal energy case and is the main estimate. It follows by an elliptic estimate and a positive commutator argument. Section \ref{highenergyoutline} contains an outline of the proof of \eqref{highenergy}, and sections \ref{ellipticsection} and \ref{mainestimateproof} contain proofs of the subsidiary estimates in the proof of \eqref{highenergy}. Appendix \ref{pseudosection} contains some important facts about pseudodifferential operators with finite regularity symbols.

The following is a frequently invoked and important estimate.
\begin{lemma} 
For any $\beta \in \Rb, q>0$ and $u,f$ solving \eqref{1dDWE} 
\begin{equation}\label{dampest}
\int W |u|^2 \leq q^{-1} \int |f u|. 
\end{equation}
\end{lemma}
\begin{proof}
Multiply \eqref{1dDWE} by $\bar{u}$ then take the imaginary part, integrating by parts to see that the term $\<\Delta u, u\>$ is real. 
\end{proof}

\section{Proof of low horizontal energy estimate \eqref{lowenergy}}\label{lowenergysection}
\begin{proof}
To prove \eqref{lowenergy} multiply \eqref{1dDWE} by $\bar{u}$ and a nonnegative function $b_{\e_1} \in \Ci(\Rb/ 2\pi \Zb)$ with
$$
b_{\e_1}(x) = \begin{cases}
\cos\left( \frac{\pi}{2(\sigma+\e_1)}x \right) & |x| < \sigma + \e_1/2, \\
0 &|x| > \sigma+\e_1.
\end{cases}
$$
Then integrate and take the real part to obtain
$$
-Re \int b_{\e_1} u'' \bar{u} - \beta \int b_{\e_1} |u|^2 = Re \int b_{\e_1} f \bar{u}.
$$
Integrating by parts once gives
$$
 \int  b_{\e_1}  |u'|^2  +Re  \int u b_{\e_1}' \bar{u}' - \beta \int b_{\e_1} |u|^2 = Re \int b_{\e_1} f \bar{u}.
$$
Integrating by parts the $u b_{\e_1}' \bar{u}'$ term again and taking advantage of the $Re$ gives 
\begin{equation}\label{lowenergyadd}
\int b_{\e_1} |u'|^2 + \int \left( - \frac{b_{\e_1}''}{2} - \beta b_{\e_1}\right)|u|^2 = Re \int b_{\e_1} f \bar{u}.
\end{equation}
Now note that $-\frac{b_{\e_1}''}{2}=\frac{\pi^2}{8(\sigma+\e_1)^2} b_{\e_1}$ for $|x|<\sigma+\frac{\e_1}{2}$. Thus for $\beta < \frac{\pi^2}{16(\sigma+\e_1)^2}$ 
$$-\frac{b_{\e_1}}{2} - \beta b_{\e_1}>c \text{ on } |x|<\sigma+\frac{\e_1}{2}.$$
So adding a multiple of \eqref{dampest}, the damping estimate, to \eqref{lowenergyadd} gives 
$$
\int |u|^2 \leq \left(C_{\e_1}+\frac{1}{q}\right) \int |f u| \leq C_{\e_1} \left( \int|f|^2 \right)^{1/2} \left( \int |u|^2 \right)^{1/2}.
$$
Dividing both sides by $\left(\int |u|^2\right)^{1/2}$ gives exactly \eqref{lowenergy}.
\end{proof}

\section{Proof of high horizontal energy estimate \eqref{highenergy}}\label{highenergyoutline}
Now that the proof of \eqref{lowenergy} is complete for $\beta< \frac{\pi^2}{16(\sigma+ \e_1)^2}$ it remains to show \eqref{highenergy} for $\e_2< \beta \leq q^2$. This estimate will actually be assembled from estimates on second microlocalized regions of phase space, in order to do so I take a semiclassical rescaling. Let $\gamma \in \{1, 2\}$, then divide both sides of \eqref{1dDWE} by $q^{2/\gamma}$ and set $h=q^{-1/\gamma}$  
\begin{equation}\label{h1dDWE}
P u = (-h^2 \p_x^2 + i h^{2-\gamma} W - h^2 \beta) u = h^2 f. 
\end{equation}
In this rescaling the bounds $\e_2 < \beta \leq q^2$ become $\e_2 < \beta \leq h^{-2\gamma}$. Let $\tau \in [\tm, 1].$ 
Take $\sigma_1$ as specified by hypothesis 1 and divide phase space $(\Rb/2\pi\Zb)_x \times \Rb_{\xi}=\Ps$ into 3 regions:
\begin{enumerate}
	\item The set where the damping is nontrivial, $\{(x,\xi): \sigma+\sigma_1/4 < |x| < \pi\}$ 
	\item The $h$ dependent elliptic set of $P$, $\{(x,\xi): |\xi|>1.5 h^{1-\tau}\}$
	\item The propagating region, $\{(x,\xi): |x| < \sigma+\sigma_1/2 \text{ and } |\xi|<2h^{1-\tau}\}$.
\end{enumerate}
Although $\gamma$ and $\tau$ can be adjusted freely, for this proof they will have a specific relation. In particular, $(\tau, \gamma)$ will only take values in $(\tm,2), (3\tm,2), (1,1)$.

Note that in composition expansions involving symbols at scale $h^{1-\tau}$ each additional term is only $h^{\tau}$ smaller than the previous one, rather than a full power of $h$. Regardless of the values of $\gamma$ and $\tau$ there is a fixed size error terms in the following calculations must be smaller than. Because of this the number of expansion terms taken (and the number of derivatives of regularity $W$ must have) grows at least like $\frac{1}{\tau}$. $\tm$ is the smallest possible $\tau$ such that $W$ has enough regularity to achieve the desired error size. 

This behavior also clarifies why $\tau$ and $\gamma$ are separate parameters. In Proposition \ref{roll} the resolvent estimate is $\ltwo{u}^2 \leq C \frac{q^{2 \frac{\tau}{\gamma}}}{\beta^2} \ltwo{f}^2$. Because of this a larger $\gamma$ produces a better estimate without decreasing $\tau$, so no additional regularity of $W$ is required. However $\gamma$ cannot always be taken large because the estimate only applies to $\beta<q^{2\tau/\gamma}$ which will not include all of $\beta<q^2$. 

Note that in the case $\gamma=\tau=1$, this is not a second microlocalization as there is no $h$ dependence. The remainder of this section is the statement of the estimates for these regions and then a proof of the high horizontal energy case, \eqref{highenergy}, using those estimates.
The damping estimate is immediate, the elliptic estimate is proved in section \ref{ellipticsection} and the propagation estimate is proved in section \ref{mainestimateproof}.

\subsection{Damping Estimate}
This lemma gives an estimate for the size of $u$ on the set where the damping is nontrivial. 
\begin{lemma}\label{hdamplemma} For any $\beta \in \Rb, h>0$ and $u,f$ solving \eqref{h1dDWE}
\begin{equation}\label{hdampest}
\ltwo{W^{1/2} u}^2 \leq h^{\gamma} \lpfu.
\end{equation}
\end{lemma}
This follows immediately from the rescaling and \eqref{dampest}.

\subsection{Elliptic Estimate}
Throughout the paper $\Op$ refers to the Weyl quantization on the torus (see Appendix \ref{pseudosection} for more details).
These lemmas gives an estimate for the size of $u$ on the $h$ dependent elliptic set of $P$, $\{(x,\xi): |\xi|>1.5 h^{1-\tau}\}$. Note that in order for $P$ to be bounded away from zero on this set $h^2 \beta$ must be smaller than $h^{2-2\tau}$. 

Because of a technicality in the proof there are separate elliptic estimates on $ch^{\tau-1} < \xi < 2$ and $1.5<\xi$. The cause of this is that the low regularity composition result (Lemma \ref{lowregcompose} which is used in the elliptic parametrix construction) requires bounded symbols but $p=\xi^2+ih^{2-\gamma}W-h^2 \beta$ is unbounded for large $\xi$. 

This lemma provides the estimate on $ch^{\tau-1} < \xi < 2$. This estimate has additional importance as it is used multiple times in the proof of the propagation estimate to provide additional control over error terms. 
\begin{lemma}\label{ellipticlemma}
Suppose $W \in \Wkinf$ and $\tau \in [\tm,1]$.
Set $z_1 \in \Ci(\Rb)$ with 
$$
z_1(\xi) = \begin{cases} 0 & |\xi| < 1.25 \\
1 & |\xi| >1.5,
\end{cases}
$$
and set $z_2 \in \Cs(\Rb)$ with 
$$
z_2(\xi) = \begin{cases} 1 & |\xi| < 2 \\
0 & |\xi| > 3,
\end{cases}
$$
then let $z(\xi) = z_1(h^{\tau-1} \xi)z_2(\xi)$ and $Z= \Op(z(\xi))$. There exist $C, h_0>0,$ such that for $h\leq h_0, \beta$ such that $h^2 \beta < h^{2-2\tau}$, and $u,f$ solving \eqref{h1dDWE} then
\begin{equation}\label{ellipticestimate1}
\lp{Zu}{2}^2 \leq Ch^{5\tau-1}\lp{f}{2}^2 + o(h^2) \ltwo{u}^2.
\end{equation}
\end{lemma}
This lemma provides the estimate on $1<\xi$. It does not impose any regularity assumptions on $W$ nor does it have a size restriction on $\beta$
\begin{lemma}\label{ellipticlemmaweak}
Set $\ti{z} \in \Ci(\Rb)$ with 
$$
\ti{z}(\xi) = \begin{cases}
0 & |\xi|<1 \\
1 & |\xi|>1.5,
\end{cases}
$$
and let $\ti{Z}=\Op(\ti{z})$. There exist $C, h_0>0$ such that for $h \leq h_0$ and $u,f$ solving \eqref{h1dDWE} then 
\begin{equation}\label{ellipticestimate2}
\lp{\ti{Z}u}{2}^2 \leq C h^4 \ltwo{f}^2 + C h^{4-\gamma} \lpfu.
\end{equation}
\end{lemma}
Lemmas \ref{ellipticlemma} and \ref{ellipticlemmaweak} are proved in section \ref{ellipticsection}.

\subsection{Propagation Estimate} 
This lemma gives an estimate for the size of $u$ on the propagating region $\{(x,\xi): |x| < \sigma+\sigma_1/2 \text{ and } |\xi|<2h^{1-\tau}\}$.

Define $\psi \in \Cs(\Rb)$ 
$$
\psi(\xi) = \begin{cases}
1 \text{ on } |\xi|<2 \\
0 \text{ on } |\xi|>3,
\end{cases}
$$ 
and $\chi \in \Cs(-\pi,\pi)$
$$
\chi(x) = \begin{cases}
1 \text{ on } |x|<\sigma+\sigma_1/2 \\
0 \text{ on } |x|>\sigma+\sigma_1,
\end{cases}
$$
where both are chosen to have smooth square roots. 

\begin{lemma}\label{mainestimate} 
Suppose $v_j \in \Wkinf$ and fix $\tau \in [\tm,1]$, $\e_2>0$. Set $J=\Op(\chi^{1/2}(x) \psi^{1/2}(h^{\tau-1}\xi))$. There exist $C,h_0>0,$ such that if $h \leq h_0$ and $\beta$ such that $h^2 \e_2< h^2 \beta < h^{2-2\tau}$,  
then for $u,f$ solving \eqref{h1dDWE}
\begin{align}\label{bigbadestimate}
\lp{Ju}{2}^2 &\leq C \left( \bigbadestimatefirstterm \right) \lp{f}{2} \lp{u}{2} +C  \frac{h^{6\tau-\gamma-1}}{\beta} \lp{f}{2}^2 + o(1) \lp{u}{2}^2.
\end{align}
\end{lemma}
Lemma \ref{mainestimate} is proved in section \ref{mainestimateproof}. $h^2 \beta \leq h^{2-2\tau}$ is assumed in order to apply the elliptic region estimate in the proof.

\subsection{Combination of Estimates}
This subsection completes the proof of \eqref{highenergy}, the high horizontal energy estimate, using the following proposition on different regimes for $\beta, \tau$ and $\gamma$. 
\begin{proposition}\label{roll}
Suppose $W \in \Wkinf$ and fix $\tau \in [\tm, 1], \gamma \in \{1,2\}$ and $\e_2>0$. There exist $C, q_0 > 0,$ such that if  $q \geq q_0$ and $\beta$ satisfies $\e_2 \leq \beta \leq q^{2\tau/\gamma}$ then for $u$ and $f$ solving \eqref{1dDWE} 
$$
\lp{u}{2}^2 \leq C \frac{q^{\frac{2\tau}{\gamma}}}{\beta^2} \lp{f}{2}^2.
$$
\end{proposition}
The form of this estimate heps show why $\tau$ and $\gamma$ are taken as two separate parameters. Taking $\gamma=2$ produces a better estimate for all values of $\tau$, however the estimate then only applies to $\beta<q$ which does not cover the required range of $\beta<q^2$. 

This proposition will be proved using the estimates on the damped region (Lemma \ref{hdamplemma}), the elliptic region (Lemmas \ref{ellipticlemma} and \ref{ellipticlemmaweak}) and the propagating region (Lemma \ref{mainestimate}). 
\begin{proof}[Proof of Proposition \ref{roll}]
Note that  $\beta \leq q^{2\tau/\gamma}$ guarantees $h^2 \beta \leq h^{2-2\tau}$ and taking $q_0$ large enough ensures that $h$ is small enough to apply the Lemmas.
Add together \eqref{hdampest}, \eqref{ellipticestimate1}, \eqref{ellipticestimate2} and \eqref{bigbadestimate}, 
\begin{align*}
\ltwo{W^{1/2}u}^2 + \ltwo{Zu}^2 +\ltwo{\ti{Z}u}^2+ \ltwo{Ju}^2 &\leq (h^{\gamma} + Ch^{4-\gamma}+C \frac{h^{-\tau}}{\beta}) \lpfu \\
&+C ( h^{5\tau-1}+ h^{4}+ \frac{h^{6\tau-\gamma-1}}{\beta}) \ltwo{f}^2  \\
&+ (o(h^2) + o(1))\ltwo{u}^2.
\end{align*}
Since $\gamma \leq 2$ and $\tau \geq 1/2$, the $(h^{\gamma}+h^{4-\gamma}) \lpfu$ and $(h^{4}+h^{5\tau-1}) \ltwo{f}$ terms on the right hand side are automatically smaller than other terms and can be safely ignored. 

Now, rewrite the LHS as $\<(W + Z^2 +\ti{Z}^2+ J^2)u, u\>$ and use the fact that $W(x)+z(\xi)^2+\ti{z}(\xi)^2+\chi(\xi)\psi(h^{\tau-1}\xi)$ is strictly positive on $\Ps$ to bound that term from below by $c \ltwo{u}^2$. Then for $h$ small enough,  absorb all the $\ltwo{u}^2$ terms from the right hand side into the left 
$$
c \ltwo{u}^2 \leq  C \frac{h^{-\tau}}{\beta} \lpfu + C \frac{h^{6\tau-\gamma-1}}{\beta}\ltwo{f}^2.
$$
Now use Young's inequality on the $\frac{h^{-\tau}}{\beta}\lpfu$ term and group the resultant $\lp{u}{2}^2$ onto the left hand side to obtain 
$$
\lp{u}{2} \leq C \left( \frac{h^{-2\tau}}{\beta^2} + \frac{h^{6\tau-\gamma-1}}{\beta} \right)\lp{f}{2} \leq C \frac{h^{-2\tau}}{\beta^2} \ltwo{f}.
$$
Where the second inequality follows because $\tau>\tm >1/2$ and $\gamma \leq 2$ imply $6\tau-\gamma -1 > 3 - \gamma -1 >0$, so the second term goes to 0 as $h \ra 0$ regardless of $\beta$.

Finally the rescaling $q=1/h^{\gamma}$ gives the desired inequality. 
\end{proof}

To finish the proof of \eqref{highenergy} it is necessary to consider different regimes for $\beta, \tau$ and $\gamma$ in order to ensure that $\beta \leq q^{2\tau/\gamma}$ and to obtain the best possible estimate. Suppose $q \geq q_0$ and consider three cases for $\e_2 \leq \beta \leq q^2$ (recalling that $1/2 \leq \tm \leq 1$)
\begin{enumerate}
	\item $\e_2 \leq \beta \leq q^{\tm}$
	\item $\frac{1}{2}q^{\tm} \leq \beta \leq q^{3 \tm}$
	\item $\frac{q^{3 \tm}}{2} \leq \beta \leq q^2.$
\end{enumerate}
In case 1 choose $\tau=\tm, \gamma=2.$ Then by Proposition \ref{roll} there exists $C>0$ such that
$$
\lp{u}{2}^2 \leq C q^{\tm} \lp{f}{2}^2.
$$
In case 2 choose $\tau=3\tm, \gamma=2.$ Then by Proposition \ref{roll}, since $\frac{1}{\beta} \leq \frac{2}{q^{\tm}}$, there exists $C>0$ such that 
$$
\lp{u}{2}^2 \leq C  \frac{q^{3\tm}}{q^{2\tm}} \lp{f}{2}^2 \leq C q^{\tm} \lp{f}{2}^2.
$$
If $3\tm  \geq 2$, skip case 3 and for case 2 instead take $\tau=1, \gamma=1$. Then by Proposition \ref{roll}, since $\frac{1}{\beta} \leq \frac{2}{q^{\tm}}$, there exists $C>0$ such that 
$$
\lp{u}{2}^2 \leq C  \frac{q^2}{q^{2\tm}} \lp{f}{2} \leq C q^{2- 2\tm} \lp{f}{2}^2 \leq C q^{\tm} \lp{f}{2}^2.
$$
where the final inequality follows since $3\tm \geq 2$ implies $2-2\tm \leq \tm$.

In case 3 choose $\tau=1, \gamma=1$. Then by  Proposition \ref{roll}, since $\frac{1}{\beta} \leq \frac{2}{q^{3\tm}}$, there exists $C>0$ such that 
$$
\lp{u}{2}^2 \leq C \frac{q^2}{q^{6\tm}} \lp{f}{2}^2 \leq C q^{2-6\tm} \lp{f}{2}^2 \leq C q^{\tm} \lp{f}{2}^2,
$$
where the final inequality holds because $\tm>1/2>2/7$ implies $2-6\tm \leq \tm$. 

Since all $q, \beta$ such that $\e_2 \leq \beta \leq q^2$ are covered by these three cases this proves the high energy estimate \eqref{highenergy}. This along with the low energy estimate \eqref{lowenergy} completes the proof of Theorem \ref{maintheorem}. 

So it now remains to prove the elliptic estimates (Lemmas \ref{ellipticlemma} and \ref{ellipticlemmaweak}) and the propagating estimate (Lemma \ref{mainestimate}). They are proved in sections \ref{ellipticsection} and \ref{mainestimateproof}, respectively. 

\section{Proof of elliptic region estimates Lemmas \ref{ellipticlemma} and \ref{ellipticlemmaweak}}\label{ellipticsection}
If $W$ is smooth and $\tau-1$ then a conventional semiclassical parametrix argument produces the desired elliptic estimate (see for example \cite{DyatlovZworski2020} Proposition E.32). Normally as part of that proof $\xi^2$ is composed with $1/p$. This becomes an issue when $W$ is not smooth as the low regularity composition expansion (Lemma \ref{lowregcompose}) only works with bounded symbols. To address this cutoff functions are used to split the estimate into estimates on a bounded elliptic set (Lemma \ref{ellipticlemma}) and a standard elliptic set (Lemma \ref{ellipticlemmaweak}). 

Taking $\tau \neq 1$ produces additional issues. The bounded elliptic set is $h$ dependent and $\xi$ is only bounded from below by a power of $h$, rather than a constant. In order to ensure that $p=\xi^2 +ih^{2-\gamma} W -h^2 \beta$ is invertible on this set $\beta$ must satisfy $h^2 \beta \leq h^{2-2\tau} < c\xi^2$. Therefore $p$ is only bounded from below by a power of $h$ and every division by $p$ creates unfavorable powers of $h$. These unfavorable powers can be controlled, but this requires additional regularity of $W$ and the requirements grow as $\tau$ approaches $1/2$.

In this section I will first prove the estimate on the $h$ dependent elliptic set (Lemma \ref{ellipticlemma}) and then prove the estimate on the standard elliptic set (Lemma \ref{ellipticlemmaweak}).

\subsection{$h$ dependent elliptic estimate, Lemma \ref{ellipticlemma}}
The proof of Lemma \ref{ellipticlemma} follows the conventional semiclassical parametrix argument with adjustments made to handle the issues described above. 

The first change is that the parametrix is constructed for a cutoff version of $P$, $\Op(\chi p)$, which is bounded.  

Let $\chi \in \Cs(\Rb)$ have
$$
\chi(\xi) = \begin{cases}
1 & |\xi| < 3.5 \\
0 & |\xi|>4.
\end{cases}
$$

Define
$$
q_0(x,\xi) = \frac{h^{2-2\tau} z(\xi)}{\chi(\xi)(\xi^2 + i h^{2-\gamma} W - h^2 \beta)} = \frac{h^{2-2\tau} z(\xi)}{\chi(\xi) p(x,\xi)} = \frac{h^{2-2\tau} z(\xi)}{p(x,\xi)},
$$
where $\chi$ can be replaced by 1 because $\chi \equiv 1$ on $supp z \subset \{ 1.5 h^{1-\tau} < |\xi| < 3\}$. The $\chi$ is included to simplify the composition with $\Op(\chi p)$. 

For $j \geq 1,$ recursively define 
$$
q_j = -\frac{1}{\chi p} \sum_{l=0}^{j-1} q_{l,j-l},
$$
where $q_{l,j-l}$ is the $(j-l)$th term in the composition expansion $\Op(q_l) \Op(\chi p)$ and is given by 
$$
q_{l,j-l} = C_{j,l}h^{j-l} (\p_x^{j-l} (\chi p) \p_{\xi}^{j-l} q_l + (-1)^{j-l} \p_{\xi}^{j-l} (\chi p) \p_x^{j-l} q_l),
$$
for $j-l \geq 1$. Once again the $\chi$ can be replaced by 1 in the definitions of $q_j$ and $q_{l,j-l}$, since $\chi$ is identically 1 on the support of $z$ and thus $q_j$. 

These $q_j$ are used to construct a parametrix for $\Op(\chi p)$, which in turn is used to control $Q_j \Op(\chi) P$. In particular, I will show that for $N$ large enough
$$
\sum_{j=0}^N Q_j \Op(\chi p) =\sum_{j=0}^N Q_j \Op(\chi) P= h^{2-2\tau} Z + o(h^{3-2\tau}).
$$
To prove this, the $q_j$ are first shown to be in a particular symbol class (Lemmas \ref{q0symbol} and \ref{qjsymbol}), which gives control of the size of $Q_j$ as operators on $L^2$ (Lemma \ref{qjl2}). Then $\sum Q_j$ and $\Op(\chi p)$ are composed in two different ways, producing error terms of the appropriate size (Lemmas \ref{qjcomposechip} and \ref{qjcomposechip2}). Finally, this composition formula is applied to $u$ which gives the desired elliptic estimate. 

The following symbol style estimates for $p$ and $\frac{W}{p}$ on $\supp z = \{1.5 h^{1-\tau} < |\xi| < 3\}$ are needed to prove the symbol estimates for $q_j$. 
\begin{lemma}\label{psymbolest}\, \\
\begin{enumerate}
\item For $m,j, v \in \Nb$ such that $2m\geq 2j$ 
$$
\sup_{(x,\xi) \in \supp z} \left| \frac{\p_{\xi}^v (p(x,\xi)^j)}{p(x,\xi)^m} \right| \leq C h^{(\tau-1)(2m-2j+v)}.
$$
\item For $\alpha, j ,t \in \Nb$ with $j \geq t$ 
$$
\sup_{(x,\xi) \in \supp z} \left| h^{(1-\tau) 2t} \frac{\p_x^{\alpha} (i h^{2-\gamma} W)^{j-t}}{p^j} \right| \leq C h^{(\tau-1) \alpha}.
$$
\end{enumerate}
\end{lemma}
\begin{proof}
1) To begin the binomial expansion formula gives 
$$
\p_{\xi}^v p(x,\xi)^j = \sum_{l=0}^j C_{l,j}  \p_{\xi}^v (\xi^2 - h^2 \beta)^l (i h^{2-\gamma} W(x))^{j-l}.
$$
Again using the binomial expansion formula  
$$
\p_{\xi}^v (\xi^2 - h^2 \beta)^l = \p_{\xi}^v \sum_{k=0}^l C_{k,l} \xi^{2k} (-h^2 \beta)^{l-k} = \sum_{k=0}^l C_{k,v,l} \xi^{2k-v} (-h^2 \beta)^{l-k}.
$$
Now note that on $\supp z,$ $|\xi|>1.5 h^{1-\tau}$ while $h^2 \beta \leq h^{2 - 2\tau}$ and so $h^2 \beta \leq \xi^2/2$. Therefore
$$
|\p_{\xi}^v (\xi^2 - h^2 \beta)^l| \leq \sum_{k=0}^l C |\xi|^{2k-v} |h^2 \beta|^{l-k} \leq \sum_{k=0}^l C |\xi|^{2l-v} \leq C |\xi|^{2l-v}.
$$

Now split $\supp z$ into two sets  
\begin{enumerate}
	\item $\Ac = \{ (x,\xi) \in \supp z; h^{2-\gamma} W(x) \leq \xi^2 \}$ 
	\item $\Bc = \{(x,\xi) \in \supp z; h^{2-\gamma} W(x) \geq \xi^2\}$
\end{enumerate}
For $(x, \xi) \in \Ac$ 
$$
|\p_{\xi}^v p^j |\leq \sum_{l=0}^j |\p_{\xi}^v (\xi^2 -h^2 \beta)^l| |h^{2-\gamma} W|^{j-l} \leq \sum_{l=0}^j |\xi|^{2l-v} |\xi^2|^{j-l} \leq C |\xi|^{2j-v}.
$$
Also $|p| = \sqrt{(\xi^2-h^2 \beta)^2 + (h^{2-\gamma} W)^2} \geq \xi^2 -h^2 \beta \geq c \xi^2$. 

Therefore
$$
\sup_{(x,\xi) \in \Ac} \left| \frac{\p_{\xi}^v p^j}{p^m} \right| \leq \sup_{(x,\xi) \in \Ac}  \frac{C |\xi|^{2j-v}}{|p|^m} \leq \sup_{(x,\xi) \in \Ac}  C |\xi|^{2j-v-2m} \leq C h^{(1-\tau)(2m-2j+v)},
$$
where the last inequality follows since $\xi > 1.5 h^{1-\tau}$ on $\supp z$. This is the desired inequality. 

Now consider the second case when $(x,\xi) \in \Bc$ i.e. $h^{2-\gamma} W \geq \xi^2$. Then $\p_{\xi}^v p \leq \sum_{l=0}^j |\xi|^{2l-v} |h^{2-\gamma} W|^{j-l} \leq |h^{2-\gamma} W|^j$ and $|p| \geq h^{2-\gamma} W$. Therefore
$$
\sup_{(x,\xi) \in \Bc} \left| \frac{\p_{\xi}^v p^j}{p^m} \right| \leq \sup_{(x,\xi) \in \Bc} \frac{W^j }{(h^{2-\gamma} W)^m} \leq  \sup_{(x,\xi) \in \Bc} \frac{1}{\xi^{2(m-j)}} \leq h^{(\tau-1)(2m-2j)} \leq h^{(\tau-1)(2m-2j+v)}
$$
since $\tau-1 \leq 0$ and $\xi>1.5 h^{1-\tau}$ on $\supp z$. These two cases cover all of $\supp z$ and so the desired inequality holds. 

2) When $2j-2t \leq \alpha$ this is true as $|p|^j \geq h^{-2j(\tau-1)}$ and so 
$$
\frac{h^{(\tau-1)(-2t)}}{|p|^j} |\p_x^{\alpha} (ih^{2-\gamma} W)^{j-t}| \leq C h^{(\tau-1)(-2t +2j)} \leq C h^{(\tau-1)\alpha}.
$$

So now assume $2j -2t> \alpha$. Applying $\p_x^{\alpha}$ to $(ih^{2-\gamma} W)^{j-t}$ produces a sum of powers of derivatives of $W$. In particular letting $j_0, j_1, \ldots j_{\alpha} \in \Nb$ then
$$
\p_x^{\alpha} W^{j-t} = \sum C_{j_0, \ldots, j_{\alpha}} W^{j_0} (\p_x W)^{j_1} (\p_x^2 W)^{j_2} \cdots (\p_x^{\alpha} W)^{j_{\alpha}},
$$
where the sum is taken over $j_0, \ldots, j_{\alpha}$ such that $j_0 + j_1 + \cdots + j_{\alpha} = j-t$ and $j_1 + 2 j_2 + \cdots + \alpha j_{\alpha} = \alpha$. These conditions guarantee that there are $j-t$ factors of $W$ on the right hand side and that each term in the sum has $\alpha$ total derivatives. 

Rearranging the derivative equation and then substituting in a rearranged version of the $W$ powers equation gives 
\begin{align*}
\alpha - j_1 = 2 j_2 + 3 j_3 + \cdots + \alpha j_{\alpha} \geq 2(j_2 + j_3 + \cdots + j_{\alpha}) = 2(j-t - j_0 -j_1). 
\end{align*}
Therefore $0<2(j-t) -\alpha \leq 2j_0 + j_1$. That is the number of terms with no derivatives or one derivative is somehow bounded from below. 

Now note that since $p=\xi^2 + ih^{2-\gamma} W - h^2 \beta$ then $|p| \geq h^{2-\gamma} W$ and so $\left|\frac{h^{2-\gamma} W}{p} \right| \leq C$. Similarly since $W=\sum v_j^2$ and the $v_j$ are bounded and have bounded derivatives then $|\p_x W| \leq C W^{1/2}$ and so $\left| \frac{h^{2-\gamma} \p_x W}{p^{1/2}} \right| \leq \frac{h^{2-\gamma} W^{1/2}}{|p|^{1/2}} \leq C$.

Therefore, again taking the sum over $j_0, j_1, \ldots, j_{\alpha}$ satisfying the derivative and powers of $W$ constraints, and using powers of $W$ and $\p_x W$ to cancel powers and half powers of $p$ respectively
$$
\left| \frac{\p_x^{\alpha} (ih^{2-\gamma} W)^{j-t}}{p^j} \right| \leq \sum C_{j_0, \ldots, j_{\alpha}}  \frac{h^{(2-\gamma)(j-t)} W^{j_0} |\p_x W|^{j_1} |\p_x^2 W|^{j_2} \cdots |\p_x W|^{j_{\alpha}}}{|p|^j} \leq \sum \frac{C}{|p|^{j-j_0-\frac{j_1}{2}}}.
$$
Now using that $|p| \geq c \xi^2$ and $|\xi| > c h^{1-\tau}$ on $\supp z$ the above equation gives 
\begin{align*}
\sup_{(x,\xi) \in \supp z} \left| h^{(1-\tau)2t} \frac{ \p_x^{\alpha} (ih^{2-\gamma} W)^{j-t}}{p^j} \right| &\leq \sup_{(x,\xi) \in \supp z}  \frac{Ch^{(1-\tau)2t}  }{|p|^{j-j_0-\frac{j_1}{2}}} \leq \sup_{(x,\xi) \in \supp z}  \frac{C h^{(1-\tau) 2t}}{|\xi|^{2j-2j_0 -j_1}}\\
& \leq h^{(1-\tau) 2t} h^{(\tau-1)(2j - 2j_0 -j_1)} = C h^{(\tau-1)(2j-2t-2j_0-j_1)} \leq C h^{(\tau-1) \alpha},
\end{align*}
where the final inequality follows from $2j_0 +j_1 \geq 2(j-t) - \alpha$. 
\end{proof}
The proof of part 2 of this lemma is a key usage of the fact that $|\nabla W| \leq W^{1/2}$ in this paper. Because of this it is worth mentioning that this exact argument does not give a meaningful improvement when $|\nabla W| \leq C W^{1-\e}$ for $\e<1/2$. With such an assumption there is still no improvement from factors of $W$ without any derivatives, the improvement can only come from factors of $\p_x W$. However when $\alpha$ is even there are always terms with $j_1=0$ with no improvement over the stated result.

With these symbol style estimates it is now possible to give the symbol class for $q_0$. Unlike other symbols in this paper, differentiating it in either $\xi$ or $x$ produces factors of $h^{\tau-1}$.
\begin{lemma}\label{q0symbol}
$$
q_0 \in W^{k_0} S_{1-\tau, 1-\tau}(\Ps).
$$
\end{lemma}
\begin{proof}
Since $\supp z \subset \{1.5 h^{1-\tau} < |\xi| < 3\}$ it is enough to show for $|\xi|<3$ and $\theta \in \Nb, |\alpha| \leq k_0$ that
$$
|\p_x^{\alpha} \p_{\xi}^{\theta} q_0| \leq h^{(\tau-1)(\alpha + \theta)}.
$$

To begin, recall a classical fact about higher order derivatives of a quotient. 
\begin{equation}\label{quotientrule}
\p_x^{\alpha} \left(\frac{f(x)}{g(x)}\right) = \sum_{k=0}^{\alpha} \sum_{j=0}^k (-1)^j \binom{\alpha}{k} \binom{k+1}{j+1} \frac{1}{g^{j+1}} \p_x^{\alpha-k} f \p_x^k g^j.
\end{equation}
This follows from the Leibniz rule and the Hoppe formula applied to $1/g$, (for the Hoppe formula see \cite{JohnsondiBruno} (3.3)) 
 
Therefore 
$$
\p_x^{\alpha} q_0(x,\xi) = \p_x^{\alpha} \left(\frac{ h^{2-2\tau} z(\xi)}{p(x,\xi)} \right)= \sum_{j=0}^{\alpha} (-1)^j \binom{\alpha+1}{j+1} \frac{h^{2-2\tau}z(\xi)}{p^{j+1}} \p_x^{\alpha} p^j.
$$
And so 
$$
\p_{\xi}^{\theta} \p_x^{\alpha} q_0 = \sum_{j=0}^{\alpha}  (-1)^j \binom{\alpha+1}{j+1} h^{2-2\tau} \p_{\xi}^{\theta}\left( \frac{z}{p^{j+1}}  \p_x^{\alpha} p^j\right).
$$
Now applying \eqref{quotientrule} to $\p_{\xi}^{\theta} \left( \frac{z}{p^{j+1}}  \p_x^{\alpha} p^j\right)$
$$
\p_{\xi}^{\theta} \p_x^{\alpha} q_0 = \sum_{j=0}^{\alpha}  \sum_{v=0}^{\theta} \sum_{w=0}^v  C_{j,\alpha, v, \theta, w}  \frac{h^{2-2\tau}}{p^{(j+1)(w+1)}} \p_{\xi}^{\theta-v} (z \p_x^{\alpha} p^j) \p_{\xi}^v p^{(j+1)w}.
$$
So it is sufficient to control each individual term in the sum, which is of the form
\begin{equation}\label{q0sumterm}
C h^{2-2\tau} \frac{1}{p^{(j+1)(w+1)}} \p_{\xi}^{\theta-v} (z \p_x^{\alpha} p^j) \p_{\xi}^v p^{(j+1)w},
\end{equation}
where $0\leq j \leq \alpha$,  $0\leq v \leq \theta$ and $0 \leq w \leq v$. 

Well by Lemma \ref{psymbolest}
$$
\sup_{(x,\xi) \in \supp z} \left| \frac{\p_{\xi}^v (p^{j+1})^w}{(p^{j+1})^w} \right| \leq O(h^{(\tau-1)v}).
$$

This along with $|p| \geq h^{2-2\tau}$ gives
$$
\sup_{(x,\xi) \in \supp z} \left|h^{2-2\tau} \frac{1}{p^{(j+1)(w+1)}} \p_{\xi}^{\theta-v} (z \p_x^{\alpha} p^j) \p_{\xi}^v p^{(j+1)w} \right| \leq \sup_{(x,\xi) \in \supp z} C \left| \frac{1}{p^{j}} \p_{\xi}^{\theta-v}(z \p_x^{\alpha} p^j) h^{(\tau-1) v} \right|
$$
Now use the product rule to expand out $\p_{\xi}^{\theta -v} (z \p_x^{\alpha} p^j)$ 
$$
C \left| \frac{1}{p^{j}} \p_{\xi}^{\theta-v}(z \p_x^{\alpha} p^j) h^{(\tau-1) v} \right| \leq \sum_{l=0}^v \left| \frac{1}{p^j}  \p_{\xi}^{\theta-v-l} z \p_{\xi}^l \p_x^{\alpha} p^j  \right|h^{(\tau-1) v} .
$$
Well $|\p_{\xi}^{\theta-v-l} z| \leq C h^{(\tau-1)(\theta-v-l)}$ which gives 
$$
\sum_{l=0}^v \left| \frac{1}{p^j}  \p_{\xi}^{\theta-v-l} z \p_{\xi}^l \p_x^{\alpha} p^j h^{(\tau-1) v} \right| \leq \sum_{l=0}^v \left| \frac{1}{p^j} h^{(\tau-1)(\theta-l)} \p_{\xi}^l \p_x^{\alpha} p^j \right|.
$$
Again use the binomial expansion to write
$$
p^j = \sum_{t=0}^j C_{j,t}( \xi^2 - h^2 \beta)^t (ih^{2-\gamma} W)^{j-t}, 
$$
and so 
$$
\p_{\xi}^l \p_x^{\alpha} p^j = \sum_{t=0}^{j} C_{j,t} \p_{\xi}^l (\xi^2 - h^2 \beta)^t \p_x^{\alpha} (ih^{2-\gamma} W)^{j-t}.
$$
Combining this with the previous chain of inequalities and \eqref{q0sumterm} gives
\begin{align*}
&\sup_{(x,\xi) \in \supp z} \left|  \frac{h^{2-2\tau}}{p^{(j+1)(w+1)}} \p_{\xi}^{\theta-v} (z \p_x^{\alpha} p^j) \p_{\xi}^v p^{(j+1)w} \right| \leq \\
&\sup_{(x,\xi) \in \supp z} \sum_{l=0}^v \sum_{t=0}^j  \left| \frac{Ch^{(\tau-1)(\theta-l)}}{p^j}  \p_{\xi}^l (\xi^2 - h^2 \beta)^t \p_x^{\alpha} (ih^{2-\gamma} W)^{j-t} \right|.
\end{align*}
By the same argument used in part 1 of Lemma \ref{psymbolest} 
$$
\sup_{(x,\xi) \in \supp z} \left| \p_{\xi}^l (\xi^2 - h^2 \beta)^t \right| \leq C (h^{\tau-1})^{l - 2t}.
$$

Therefore
\begin{align*}
\sup_{(x,\xi) \in \supp z} \sum_{l=0}^v \sum_{t=0}^j  &\left| \frac{Ch^{(\tau-1)(\theta-l)}}{p^j}  \p_{\xi}^l (\xi^2 - h^2 \beta)^t \p_x^{\alpha} (ih^{2-\gamma} W)^{j-t} \right| \\
\leq \sup_{(x,\xi) \in \supp z}  \sum_{t=0}^j  &\left| \frac{C h^{(\tau-1) (\theta-2t)}}{p^j} \p_x^{\alpha} (ih^{2-\gamma} W)^{j-t} \right|.
\end{align*}
Now using part 2 of Lemma \ref{psymbolest}
$$
\sup_{(x,\xi) \in \supp z} \left| \frac{h^{(\tau-1)\theta}}{p^j} h^{(1-\tau) 2 t} \p_x^{\alpha}(ih^{2-\gamma} W)^{j-t} \right| \leq C h^{(\tau-1)(\theta+\alpha)},
$$
and combining this chain of inequalities gives the desired statement.  
\end{proof}

Next I show that $q_j$ is a symbol with the same behavior under differentiation by $x$ and $\xi$ but with size $h^{j(2\tau-1)}$. 
\begin{lemma}\label{qjsymbol}
$$
h^{j(1-2\tau)} q_j \in W^{k_0-j} S_{1-\tau, 1-\tau}(\Ps)
$$
\end{lemma}
\begin{proof}
Since $\supp z \subset \{1.5 h^{1-\tau} < |\xi| < 3\}$ it is enough to show for $|\xi|<3$ and $\theta \in \Nb, \alpha \leq k_0-j$ that
$$
|\p_x^{\alpha} \p_{\xi}^{\theta} q_j| \leq C h^{j(2\tau-1)} h^{(\tau-1)(\alpha+\theta)}.
$$
This will be proved inductively in $j$. By Lemma \ref{q0symbol}, $q_0$ satisfies this. So assume $|\p_x^{\alpha} \p_{\xi}^{\theta} q_l| \leq Ch^{l(2\tau-1)} h^{(\tau-1)(\alpha+\theta)}$ for all $0 \leq l \leq k$ and it is enough to show for $\theta \in \Nb, \alpha \leq k_0-j$ 
$$
\sup_{|\xi| <3} |\p_x^{\alpha} \p_{\xi}^{\theta} q_{k+1} | \leq C h^{(k+1)(2\tau-1)} h^{(\tau-1)(\alpha+\theta)}.
$$ 
By definition
$$
q_{k+1} = - \frac{1}{p} \sum_{l=0}^{k} q_{l,k+1-l} = \frac{-q_{k,1}}{p} - \frac{1}{p} \sum_{l=0}^{k-1} q_{l,k+1}. 
$$
Since 
$$
q_{l,k+1-l} = Ch^{k+1-l} (\p_x^{k+1-l} p \p_{\xi}^{k+1-l} q_l + (-1)^{k+1-l} \p_{\xi}^{k+1-l} p \p_x^{k+1-l} q_l )
$$
then by the inductive assumption
$$
\sup_{|\xi| <3} |q_{l,k+1-l}| \leq \sup_{|\xi| <3} C h^{(k+1-l)\tau} |q_l|  \leq C h^{(k+1-l)\tau} h^{l(2\tau-1)}. 
$$
Therefore 
$$
\left| \frac{q_{l,k+1-l}}{p} \right| \leq Ch^{2\tau-2} h^{(k+1-l) \tau} h^{l(2\tau-1)} =C h^{2 \tau-2} h^{(k+1)\tau} h^{l(\tau-1)}. 
$$
Now for $0 \leq l \leq k-1$, since $\tau-1\leq0$. 
$$
h^{2 \tau - 2} h^{\tau(k+1)} h^{l(\tau-1)} \leq h^{2 \tau -2} h^{\tau(k+1)} h^{(k-1)(\tau-1)}.
$$
Note that
$$
2 \tau -2 + \tau k + \tau + k \tau - \tau - k +1 = 2 \tau k + 2 \tau - (k+1) = (2\tau-1)(k+1).
$$
Therefore 
$$
|q_{k+1}| \leq |\frac{q_{k,1}}{p}| + Ch^{(k+1)(2\tau-1)}.
$$
The $q_{k,1}$ term requires separate treatment. Recall $q_{k,1}= h(\p_x p \p_{\xi} q_k - \p_{\xi} p \p_x q_k)$. Using the arguments of part 2 of Lemma \ref{psymbolest} $|\frac{\p_x p}{p}|=|\frac{\p_x W}{p}|\leq C h^{\tau-1}$ and  $|\frac{\p_{\xi} p}{p}| \leq C h^{\tau-1}$ on $\supp z$ and by the inductive assumption $|\p_x q_k|, |\p_{\xi} q_k| \leq C h^{(\tau-1)} h^{k(2\tau-1)}$. Therefore 
$$
\sup_{|\xi|<3} \left| \frac{q_{k,1}}{p} \right| \leq \sup_{|\xi|<3}  C h \left( \left|\frac{\p_x p}{p}\right| |\p_{\xi} q_k| + \left|\frac{\p_{\xi} p}{p}\right| + |\p_x q_k| \right) \leq C h h^{\tau-1} h^{\tau-1} h^{k(2\tau-1)} = Ch^{(k+1)(2\tau-1)}. 
$$
It remains to be seen that $q_{k+1}$ has the correct behavior under differentiation. That is $|\p_x^{\alpha} \p_{\xi}^{\theta} q_{k+1}| \leq C h^{(k+1)(2\tau-1)} h^{(\tau-1)(\alpha+\theta)}$. Well
$$
\p_x^{\alpha} \p_{\xi}^{\theta} q_{k+1} = -\sum_{l=0}^k \p_x^{\alpha} \p_{\xi}^{\theta} \frac{q_{l,k+1-l}}{p} =- \sum_{l=0}^k \p_x^{\alpha} \p_{\xi}^{\theta} \left( \frac{h^{k+1-l}}{p} ( \p_x^{k+1-l} p \p_{\xi}^{k+1-l} q_l + (-1)^{k+1-l} \p_{\xi}^{k+1-l} p \p_x^{k+1-l} q_l ) \right)
$$
If a derivative falls on $\p_x p \p_{\xi} q$ or $\p_{\xi} p \p_{x} q$ this only produces an additional $h^{\tau-1}$. Furthermore, by the argument of Lemma \ref{q0symbol} any derivatives which fall on $\frac{1}{p}$ produce only $h^{\tau-1}$. Therefore
$$
|\p_x^{\alpha} \p_{\xi}^{\theta} q_{k+1}| \leq C h^{(\tau-1)(\alpha+\theta)} |q_{k+1}| \leq C h^{(\tau-1)(\alpha+\theta)} h^{(k+1)(2\tau-1)}. 
$$
This is exactly the desired inductive statement, which completes the proof. 
\end{proof}
With these symbol estimates it is straightforward to control the size of $\Op(q_j)$ on $L^2$. 
\begin{lemma}\label{qjl2}
$$
\| \Op(q_j) \|_{L^2 \ra L^2} = Ch^{j(2\tau-1)} h^{\tau-1}
$$
\end{lemma}
\begin{proof}
This follows immediately from Lemma \ref{calderonvaillancourt} and Lemma \ref{qjsymbol}. In particular
\begin{align*}
\| \Op(q_j) \|_{L^2 \ra L^2}  &\leq C \sum_{\alpha,\theta \in \{0,1\}} h^{\theta} \lp{\p_x^{\alpha} \p_{\xi}^{\theta} q_j}{\infty} \leq \sum_{\alpha,\theta \in \{0,1\}} C h^{j(2\tau-1)} h^{(\tau-1)(\alpha+\theta)} h^{\theta} \\
&= \sum_{\alpha,\theta \in \{0,1\}} Ch^{j(2\tau-1)} h^{(\tau-1)\alpha} h^{\theta\tau} \leq C h^{j(2\tau-1)} h^{\tau-1}.
\end{align*}
\end{proof}

Now using these symbol and operator norm estimates it is possible to compute the composition of $\sum Q_j$ with $\Op(\chi p)$. 
\begin{lemma}\label{qjcomposechip}
If $W \in \Wkinf$ and $\tau \in [\tm, 1]$, 
$$
\left(\sum_{j=0}^{k_0-6} Q_j \right) \Op(\chi p) = h^{2-2\tau} Z + o(h^{3-2\tau}). 
$$
\end{lemma}
\begin{proof}
Applying Lemma \ref{lowregcompose} part 3 to the composition $Q_j \Op(\chi p)$, (with $N=k_0-j$ and $\tilde{N}=k_0-j-5$) produces 
$$
Q_j \Op(\chi p) = \sum_{k=0}^{k_0-6} \Op(q_{j,k}) + O_{L^2 \ra L^2}(h^{j(2\tau-1)} h^{\tau(k_0-j-5)} h^{5(\tau-1)}).
$$
The additional $h^{j(2\tau-1)}$ in the remainder term comes from the fact that $h^{j(1-2\tau)} q_j \in W^{k_0} S_{1-\tau, 1-\tau}$. 

Now, to control the remainder term, since $j \leq k_0-6$ and $\tau-1<0$
$$
j(2\tau-1) + \tau(k_0-j-5) + 5(\tau-1) = \tau k_0 + j(\tau-1)-5 \geq \tau k_0 +(k_0-6)(\tau-1) - 5.
$$
Furthermore $\tau \geq \tm > \frac{k_0+2}{2k_0-4}$ and $k_0 \geq 8 \geq 8\tau$ and so
\begin{equation}\label{eqellipticremainder}
\tau k_0 + (k_0-6)(\tau-1) - 5 = (2k_0+4)\tau -10\tau +1 >k_0+2-10\tau+1= 3-2\tau + (k_0-8\tau) \geq 3-2\tau,
\end{equation}
and the remainder error term is always of size $o(h^{3-2\tau})$. 

Now summing these composition expansions from $j=0$ to $j=k_0-6$ 
$$
\left(\sum_{j=0}^{k_0-6} Q_j \right) \Op(\chi p) =\sum_{j=0}^{k_0-6} \left( \sum_{k=0}^{k_0-j-6} \Op(q_{j,k})\right)+ o(h^{3-2\tau}) 
$$
$$
=\Op \left(\begin{array}{ccccccc}
q_{0,0} & +q_{0,1} & +q_{0,2} &+ q_{0,3} & + \cdots &   + q_{0,k_0-6}\\
 & +q_{1,0} & +q_{1,1} & +q_{1,2}& + \cdots  & + q_{1,k_0-7} \\
 &  & +q_{2,0} & + q_{2,1} & + \cdots  & + q_{2,k_0-8}\\
 &  &  & + q_{3,0}& + \cdots  & + q_{3,k_0-9} \\
  & &  & &  & +\cdots \\
  & &  & &   & +q_{k_0-6,0} \\
\end{array}\right) + o(h^{3-2\tau}).
$$
By construction of the $q_{j,k}$, all columns except for the first sum to zero leaving
$$
\left(\sum_{j=0}^{M-1} Q_j \right) \Op(\chi p) = \Op(q_{0,0}) + o(h^{3-2\tau}) = \Op(q_0 \chi p) + o(h^{3-2\tau}).
$$
Since $q_0=\frac{h^{2-2\tau} z(\xi)}{\chi p}, \Op(q_0 \chi p) = \Op(z(\xi)) = Z$ and this is the desired equality. 
\end{proof}

The composition of $\sum Q_j$ with $\Op(\chi p)$ can also be computed in a way that separates $\Op(\chi)$ and $P$.
\begin{lemma}\label{qjcomposechip2}
If $W \in \Wkinf$ and $\tau \in [\tm, 1]$, 
$$
\sum_{j=1}^{k_0-6} Q_j \Op(\chi p) = \sum_{j=1}^{k_0-6} Q_j \Op(\chi) P + o(h^{3-2\tau}).
$$
\end{lemma}
\begin{proof}
First, by Lemmas \ref{exoticcompose} and \ref{lowregcompose} part 1 (with $\ti{N}=3\leq k_0-5$ and $\rho=0$))
\begin{align*}
\Op(\chi) P &= \sum_{k=0}^{2} \frac{(ih)^k}{2^k k!} \Op(\p_{\xi}^k \chi \p_x^k p) + \Olt(h^3) \\
&=\Op(\chi p) + \sum_{k=1}^{2} C_k h^k \Op(\p_{\xi}^k \chi \p_x^k W) + \Olt(h^3).
\end{align*}
Therefore 
\begin{align*}
Q_j \Op(\chi p)= Q_j \Op(\chi)P - \sum_{k=1}^2 C_k h^k Q_j \Op(\p_{\xi}^k \chi \p_x^k W) + \Olt(h^3),
\end{align*}
and it remains to control terms of the form $Q_j \Op(\p_{\xi}^k \chi \p_x^k W)$. Well by Lemma \ref{lowregcompose} part 3 (since $\p_x^k W \in W^{k_0-k} S_0$ and $h^{j(1-2\tau)} Q_j \in W^{k_0-j} S_{1-\tau,1-\tau}$ take $\ti{N}=k_0-\max(j,k)-5$)
\begin{align*}
h^k Q_j \Op(\p_{\xi}^k \chi \p_x^k W) &= h^k \sum_{l=0}^{\ti{N}-1} \frac{(ih)^l}{2^l l!} (\p_y \p_{\xi} - \p_x \p_{\eta})^l (q_j(x,\xi) \p_{\eta}^k \chi(\eta) \p_y^k W(y)) \bigg|_{y=x, \eta=\xi} \\
&+ \Olt(h^k h^{j(2\tau-1)} h^{\tau(k_0-\max(j,k)-5)} h^{5(\tau-1)}).
\end{align*}

All the terms in the sum are 0 because $\chi \equiv 1$ on $\supp z \supset \supp q_j$ and so $\chi^{(k)} \equiv 0$ on $\supp q_j$. 

The size of the remainders can also be controlled. Since $j \leq k_0-6$ 
\begin{align*}
k+j(2\tau-1) + \tau(k_0-\max(j,k) -5) + 5(\tau-1) &= \tau k_0 + j(\tau-1) - 5 + k +j \tau -\max(j,k)\tau \\
&\geq \tau k_0 + j(\tau-1)-5 \\
&\geq \tau k_0 + (k_0-6)(\tau-1) -5 > 3-2\tau,
\end{align*}
where the last inequality follows from \eqref{eqellipticremainder}. 
Therefore
$$
Q_j \Op(\chi p) = Q_j \Op(\chi) P + o(h^{3-2\tau}),
$$
and so
$$
\sum_{j=1}^{k_0-6} Q_j \Op(\chi p) = \sum_{j=1}^{k_0-6} Q_j \Op(\chi) P + o(h^{3-2\tau}),
$$
as desired.
\end{proof}
With these two composition results the proof of the $h$ dependent elliptic estimate can be completed.
\begin{proof}[Proof of Lemma \ref{ellipticlemma}]
By Lemma \ref{qjcomposechip}
\begin{align*}
h^{2-2\tau} Zu &= \left(\sum Q_j \right) \Op(\chi p) u + o(h^{3-2\tau}) u.
\end{align*}
and by Lemma \ref{qjcomposechip2}
\begin{align*}
h^{2-2\tau} Zu &= \left(\sum Q_j \right) \Op(\chi p) u + o(h^{3-2\tau}) u= \left(\sum Q_j \right) \Op(\chi)P u + o(h^{3-2\tau}) u \\
&= \left(\sum Q_j \right) \Op(\chi) h^2 f + o(h^{3-2\tau}) u.
\end{align*}
Take the $L^2$ norm squared of both sides. Then by Lemma \ref{qjl2}, the $Q_j$ are bounded by $Ch^{\tau-1}$ on $L^2$ and $\Op(\chi)$ is bounded on $L^2$ by Lemma \ref{calderonvaillancourt}.
\begin{align*}
h^{4-4\tau} \lp{Zu}{2}^2 &\leq h^4 \lp{\sum Q_j \Op(\chi) f}{2}^2 + o(h^{6-4\tau}) \lp{u}{2}^2 \\
& \leq Ch^4 h^{\tau-1} \lp{f}{2}^2 + o(h^{6-4\tau}) \lp{u}{2}^2.
\end{align*}
Finally multiply both sides by $h^{4\tau-4}$ to obtain the desired inequality. 
\end{proof}

\subsection{$\xi>c$ Elliptic estimate, Lemma \ref{ellipticlemmaweak}}
This proof follows the conventional semiclassical parametrix argument with the caveat that $W$ is treated as a perturbation. This allows the parametrix construction to be exact, as it involves only Fourier multipliers. Because of this there are no compositions involving $W$ and so the regularity of $W$ is not involved in this proof. This same construction can also be used to prove an $h$ dependent elliptic estimate, however treating $W$ perturbatively produces an error term that weakens the estimate. This lessens the improvement the estimate makes when applied to error terms in the propagation argument and would weaken the overall conclusion.

\begin{proof}
Define 
$$
\ti{q_0} = \frac{\ti{z}(\xi)}{\xi^2-h^2 \beta}.
$$
Noting that $\ti{q_0} \in S^0_0$ since $\xi>1$ on $\supp \ti{z}$ and $h^2 \beta < 1$. Now set $\ti{Q_0} = \Op(\ti{q_0})$ and let $p_0=\xi^2 -h^2 \beta, P_0 = -h^2 \p_x^2 -h^2 \beta$. Since $q_0$ and $p_0$ both depend only on $\xi$ their composition is exact 
$$
\ti{Q_0} P_0 = \Op(\ti{q_0} p_0) = \ti{Z}.
$$
Now since $P_0+ih^{2-\gamma} W = P$
\begin{align*}
\ti{Z} u &= \ti{Q_0} P u - \ti{Q_0}(ih^{2-\gamma} W u) \\
&= h^2 \ti{Q_0}f - \ti{Q_0}(ih^{2-\gamma} W u). 
\end{align*}
Take the $L^2$ norm squared of both sides then use that $\ti{Q_0}$ is bounded on $L^2$ by Lemma \ref{calderonvaillancourt}
\begin{align*}
\ltwo{\ti{Z} u}^2 &\leq h^4 \ltwo{\ti{Q_0} f}^2 + h^{4-2\gamma} \ltwo{\ti{Q_0}  Wu}^2 \\
&\leq Ch^4 \ltwo{f}^2 + C h^{4-2\gamma} \ltwo{Wu}^2.
\end{align*}
Finally use that $W^2 \leq C W$ and \eqref{hdampest}, the damped region estimate, to obtain
$$
\ltwo{\ti{Z} u}^2\leq Ch^4 \ltwo{f}^2 + Ch^{4-\gamma} \lpfu.
$$
\end{proof}
\section{Proof of propagating region estimate Lemma \ref{mainestimate}}\label{mainestimateproof}
With the elliptic and damped region estimates proved, it remains to prove the estimate for the propagating region, that is Lemma \ref{mainestimate}. The plan for this section is as follows: first the computation of a commutator in two different ways, second the estimation of terms in the computation using expansions of compositions of pseudodifferential operators. 

\begin{proof}[Proof of Lemma \ref{mainestimate}]
Set $a=x \chi(x) (\xi h^{\tau-1}) \psi(\xi h^{\tau-1})$ and $A=\Op(a)$. Note that by Lemma \ref{calderonvaillancourt}, $A$ is bounded on $L^2$. To begin, compute $h^{1-\tau} (AP-P^* A)$ in two different ways
\begin{equation}\label{startest}
h^{1-\tau} \<[h^2 \p_x^2,  A]u, u\> + i h^{3-\gamma-\tau}  \<(AW+WA)u ,u \>= h^{1-\tau} \<(AP-P^*A)u ,u \>  = 2i h^{3-\tau} Im\<f,  Au\>.
\end{equation}
This equation is the basis of the proof. The right hand side is a term of the form $Ch^{3-\tau} \lpfu$, which is the primary term in the estimate. 

On the left hand side $h^{3-\gamma-\tau} (AW+WA)$ will produce a $h^{3-\tau} \lpfu$ and two error terms: $h^{2-\gamma+6\tau} \ltwo{f}^2$ and $o(h^3) \ltwo{u}^2$.  The $ h^{1-\tau} [h^2 \p_x^2, A]$ term will provide the $h^3 \beta Ju$ term as well as another $h^{3-\tau} \lpfu$ and error terms $o(h^3) \ltwo{u}^2$ and $h^{2-\gamma+6\tau} \ltwo{f}^2$. 

Note that most of these terms have a common factor of $h^3$ which will be divided out in order to obtain the final conclusion. Because of this throughout the section error terms must be of size $o(h^3)$ to be negligible.

I will first compute the $AW$ and $[h^2 \p_x^2, A]$ terms and then use them to prove Lemma \ref{mainestimate}. Subsection \ref{awsubsec} estimates the damping anti-commutator $(AW+WA)$, subsection \ref{hdxsubsec} estimates the Laplacian commutator $[h^2 \p_x^2, A]$ and subsection \ref{synthesis} synthesizes these to complete the proof of Lemma \ref{mainestimate}. 

\textbf{Remark} In this section I write 
$$
a^{(j)}(\xi h^{\tau-1})= h^{j(1-\tau)} \p_{\xi}^j (a(x,\xi h^{\tau-1})) = h^{j(1-\tau)} x \chi(x) \p_{\xi}^j( (\xi h^{\tau-1}) \psi(\xi h^{\tau-1})).
$$
Note that $a^{(j)} \in S^{0}_{1-\tau}(\Ps)$, see Appendix \ref{pseudosection} for the definition of $\Smr$. The utility of this notation is that $h^j \p_{\xi}^j a = h^{j \tau} a^{(j)}$, which simplifies composition expansions. This agrees with the standard usage of the notation: if $\psi^{(k)} (\xi h^{\tau-1})$ is the $k$th derivative of $\psi$ evaluated at $\xi h^{1-\tau}$ then $h^{k(1-\tau)} \p_{\xi}^k (\psi(\xi h^{\tau-1}) ) =  \psi^{(k)}(\xi h^{\tau-1})$. 

Also in this section, recall that there is a fixed $\e_2>0$ and it is assumed that $h^2 \e_2 < h^2 \beta < h^{2-2\tau}$. This assumption is needed in order to apply the elliptic estimate (Lemma \ref{ellipticlemma}) in order to control the size of error terms. 

\subsection{Damping Anti-commutator estimate}\label{awsubsec}
In order to estimate $h^{3-\gamma-\tau}(AW+WA)$ I will write it as $h^{3-\gamma-\tau}v_j  A v_j$ plus error terms. The $h^{3-\gamma-\tau} v_j  A v_j$ term can be controlled using the damping estimate and is of size $h^{3-\tau}\lpfu$. The error terms are either of size $o(h^3),$ which is small enough to be negligible or are supported on the elliptic set of $P$ and can be further controlled by the elliptic estimate and Lemma \ref{commutebump}. The terms controlled with the elliptic estimate will produce the $h^{2-\gamma+6\tau} \ltwo{f}^2$. 
In particular in this subsection I will show 
\begin{equation}\label{awest}
h^{3-\gamma-\tau} |\<(AW+WA)u, u\>| \leq Ch^{3-\tau} \lp{f}{2} \lp{u}{2} + Ch^{2+6\tau-\gamma} \lp{f}{2}^2 + o(h^3)\lp{u}{2}^2.
\end{equation}
To begin recall that $W=\sum v_j^2$ so 
\begin{align*}
h^{3-\gamma-\tau} (AW+WA) &= h^{3-\gamma-\tau}  \sum_j v_j^2 A +  A v_j^2 \nonumber \\
&=\sum_j h^{3-\gamma-\tau} v_j  A v_j + \sum_j h^{2-\gamma} [[h^{1-\tau}A,v_j],v_j].
\end{align*}
To control the first term use Lemma \ref{calderonvaillancourt} to see $A$ is bounded on $L^2$
$$
h^{3-\gamma-\tau} \sum_j |\<v_j A v_j u,u\>| \leq C h^{3-\gamma-\tau} \sum_j \lp{v_ju}{2}^2.
$$
Then since $v_j^2 \leq \sum_j v_j^2 = W$ use \eqref{hdampest}
\begin{align*}
C h^{3-\gamma-\tau} \sum_j \lp{v_ju}{2}^2 = C h^{3-\gamma-\tau} \lp{W^{1/2} u }{2}^2 \leq C h^{3-\tau} \lp{f}{2} \lp{u}{2} .
\end{align*}
Combining these inequalities gives
\begin{align}\label{awop}
h^{3-\gamma-\tau} |\<(AW+WA)u,u\>| \leq Ch^{3-\tau} \lpfu + \sum_j h^{2-\gamma} |\<[[h^{1-\tau} A, v_j], v_j]u, u\>|.
\end{align}
The sum of double commutators will be error terms. Its size can be controlled using the elliptic estimate, to do so the double commutator must first be computed.
\begin{lemma}\label{doublecommutator} 
If $v_j \in W^{k_0, \infty}$ and $\tau \in [\tm, 1]$, then
$$
\doublecommutatoreq
$$
\end{lemma}
Before proving this in the finite regularity case it is useful outline the proof  when $v_j$ is smooth, as the argument is simpler but has the same structure. Fix $M> \lceil\frac{2}{\tau}+1 \rceil$ and apply Lemma \ref{exoticcompose} to compute the commutator 
\begin{align*}
[h^{1-\tau}A, v_j] &= h^{1-\tau} \sum_{k=0}^{M-1} \frac{(ih)^k}{2^k k!}  (1-(-1)^k) \Op(\p_x^k v_j \p_{\xi}^k a)  + \Olt(h^{M\tau+1-\tau}).
\end{align*}
Then apply Lemma \ref{exoticcompose} again, to compute the double commutator
\begin{align*}
[[h^{1-\tau} A, v_j], v_j] &=  h^{1-\tau}\sum_{l=0}^{M-1} \sum_{k=0}^{M-1} \frac{(ih)^{k+l}}{2^{k+l} k! l!} (1-(-1)^k)(1-(-1)^l) \Op(\p_x^k v_j \p_x^l v_j \p_{\xi}^{k+l} a) + \Olt(h^{M\tau+1-\tau}).
\end{align*}
Since $M>\lceil\frac{2}{\tau}+1 \rceil$, $M \tau + 1 - \tau > 3$. Therefore 
\begin{align*}
[[h^{1-\tau} A, v_j], v_j]= h^{1-\tau} \sum_{l=0}^{M-1} \sum_{k=0}^{M-1} \frac{(ih)^{(k+l)}}{2^{k+l} k! l!} (1-(-1)^k)(1-(-1)^l)  \Op(\p_x^k v_j \p_x^l v_j \p_{\xi}^{k+l} a) + o(h^3). 
\end{align*}
The terms with $l=0$ or $k=0$ drop out because of the factors $1-(-1)^k$ or $1-(-1)^l$. The final step is to substitute $h^{k+l} \p_{\xi}^{k+l} a = h^{\tau(k+l)} a^{(k+l)}$, which gives an expansion of the desired form. 

When $W$ is not smooth there are two changes: the remainder term is larger and additional care must be taken to track the exact number of derivatives used. The computation of the size of the remainders in this proof are the reason  $\tm> \frac{7}{k_0-1}$ is required. Other remainder size calculations in this section involving $\tau$ make use of this relationship between $\tm$ and $k_0$ but are not sharp on it. 

\begin{proof}[Proof of Lemma \ref{doublecommutator}]
Since $v_j \in W^{k_0, \infty}$ use part 2 of Lemma \ref{lowregcompose} (with $N=k_0$) to compute the commutator 
\begin{align*}
[h^{1-\tau}A, v_j] &= \sum_{k=0}^{k_0-6} \frac{(ih)^k}{2^k k!} h^{1-\tau} (1-(-1)^k)  \Op(\p_x^k v_j \p_{\xi}^k a) + \Olt(h^{(k_0-5) \tau - 5(1-\tau) + 1 -\tau}). 
\end{align*}
There are two key differences between this and the smooth case. The expansion can only be taken to the term $k_0-6$  and the remainder term has an additional $h^{-5(1-\tau)}$. These two changes are connected; in order to show that the remainder term is a bounded operator on $L^2$ the symbol must be in $W^{5, \infty}$. Those derivatives of the symbol appear in the $L^2$ operator norm of the remainder and each $\xi$ derivative of $a$ produces a factor $h^{1-\tau}$. See Appendix A for a more detailed proof and discussion. 

The relationship between the regularity of the damping, $k_0,$ and $\tau$ guarantees that the remainder term is $o(h^3)$. In particular since $\tau \geq \tm > \frac{7}{k_0-1}$. 
\begin{equation}\label{sharptm2}
(k_0-5) \tau - 5(1-\tau) + 1 -\tau=(k_0-5) \tau - 4(1-\tau) = (k_0-1)\tau -4 > (k_0-1) \frac{7}{k_0-1} - 4  = 3.
\end{equation}

Now replacing $h^k \p_{\xi}^k a = h^{\tau k} a^{(k)}$ gives
$$
[h^{1-\tau} A, v_j] = \sum_{k=0}^{k_0-6} \frac{i^k h^{k \tau}}{2^k k!} (1-(-1)^k) h^{1-\tau} \Op(\p_x^k v_j a^{(k)}) + o(h^3).
$$
To finish computing $[[A,v_j],v_j]$ Lemma \ref{lowregcompose}  will be applied again, paying special attention to the terms in the sum that are not $o(h^3)$. In particular these terms will be supported on the elliptic set of $P$ and so can be further estimated. However, these terms have derivatives of $v_j$ in them and so their regularity must be carefully tracked. 

Since $v_j \in W^{k_0, \infty}$ and $\p_x^k v_j a^{(k)} \in W^{k_0-k} S_{\rho},$ I can apply Lemma \ref{lowregcompose} part 1 with $N=k_0$  (as $k_0 -k \geq 6 > 5$) and obtain 
\begin{align*}
h^{1-\tau} h^{k \tau} [\Op(\p_x^k v_j a^{(k)}), v_j] = \sum_{l=0}^{k_0-6} &\frac{i^l}{2^l l!} h^{1-\tau} h^{\tau(k+l)} (1-(-1)^l) \Op(\p_x^k v_j \p_x^l v_j a^{(k+l)})  \\
&+ \Olt(h^{(k_0-5)\tau-5(1-\tau) + 1-\tau+k\tau}).
\end{align*}
Where I have replaced $h^{\tau k} h^{l} \p_{\xi}^l a^{(k)} = h^{\tau(k+l)} a^{(k+l)}$. The remainder term is of size  $(k_0-5+k)\tau - 5(1-\tau)+1-\tau$. Since $k\geq 0$, by the same argument as above, the remainder term is $o(h^3)$. 

So combining 
\begin{align*}
h^{1-\tau} [[A,v_j],v_j] = \sum_{k=0}^{k_0-6} &\frac{i^k}{2^k k!} (1-(-1)^k) h^{1-\tau} h^{k \tau} [\Op(\p_x^k v_j a^{(k)}), v_j] + o(h^3) \\
&=\sum_{k=0}^{k_0-6} \sum_{l=0}^{k_0-6} \frac{i^{k+l}}{2^{k+l} k! l!} (1-(-1)^k)(1-(-1)^l) h^{1-\tau} h^{(k+l)} \Op(\p_x^k v_j \p_x^l v_j a^{(k+l)}) + o(h^3). 
\end{align*}
The terms with $k=0$ or $l=0$ again vanish because of the factors $1-(-1)^k$ and $1-(-1)^l$. This gives the formula in the statement of the lemma.
\end{proof}

Now the size of $[[A,v_j],v_j]$ can be further controlled since $\Op(\p_x^k v_j \p_x^l v_j a^{(k+l)})$ has support contained in $\{2h^{1-\tau} < |\xi|< 3 h^{1-\tau}\}$, which is contained in the elliptic set of $P$. Because of this these terms can be further estimated using the elliptic estimate. 

\begin{lemma}\label{ellipticestawaw}
If $v_j \in W^{k_0, \infty}$ and $\tau \in [\tm, 1]$ then 
\begin{equation}\label{anticommutatorest}
|\<[[h^{1-\tau} A,v_j], v_j]u, u\>| \leq C h^{6\tau} \lp{f}{2}^2 + o(h^3) \ltwo{u}^2. 
\end{equation}
\end{lemma}
\begin{proof}
For this proof I will use the notation $\bkl = \p_x^k v_j \p_x^l v_j a^{(k+l)}$. In this notation Lemma \ref{doublecommutator} is
$$
[[h^{1-\tau} A, v_j], v_j] = \sum_{k=1}^{k_0-6} \sum_{l=1}^{k_0-6} C_{k,l} h^{\tau(k+l)} h^{1-\tau} \Op(b_{k,l}) + o(h^3).
$$
The elliptic estimate as written can't be applied to $\Op(\bkl)$ directly because the symbol is not smooth. Instead, I reintroduce the operator $Z$ defined in Lemma \ref{ellipticlemma}. 

The key property here is that $Z \equiv 1$ on $\{1.5 h^{1-\tau} < |\xi| < 2\} \supset \supp(\bkl)$. This along with $\bkl \in W^{k_0 - \max(k,l)}S_{1-\tau}(\Ps)$, means Lemma \ref{commutebump} can be applied with $N=k_0-\max(k,l)$ 
\begin{align*}
\Op(\bkl) &= Z \Op(\bkl) Z +  \Olt(h^{\tau (k_0-\max(k,l))-5}).
\end{align*}
That is conjugating $\Op(\bkl)$ by $Z$ is $\Op(\bkl)$ modulo an error term. Note that there is less regularity for larger $l,k$ and so the error term is larger. However after reintroducing the $h^{1-\tau} h^{(k+l) \tau}$ from the sum the error terms can be uniformly controlled
\begin{align*}
h^{1-\tau} h^{(k+l)\tau} \Op(\bkl) &= h^{1-\tau} h^{(k+l)\tau}  Z \Op(\bkl) Z + \Olt(h^{(k_0-\max(k,l))\tau-5} h^{\tau(k+l)} h^{1-\tau}).\nonumber
\end{align*}
In particular since $\tau \geq \tm >\frac{7}{k_0-1}$ and $k, l \geq 1$ 
$$
(k_0-\max(k,l) )\tau - 5 + \tau(k+l) + 1 - \tau = (k_0 + \min(k,l) -1) \tau -4 > \frac{7}{\tau} \tau -4 + \tau = 3 + \tau,
$$
and the error term is $o(h^3)$. 

Therefore
\begin{align}\label{ellipticaw}
h^{1-\tau} h^{(k+l)\tau} \Op(\bkl) &= h^{1-\tau} h^{(k+l)\tau}  Z \Op(\bkl) Z+\olt(h^{3}).
\end{align}
 
Now, apply \eqref{ellipticaw} term by term to Lemma \ref{doublecommutator} 
\begin{align*}
|\<[[h^{1-\tau} A,v_j],v_j] u, u\>| &\leq |\<Z h^{1-\tau} [[A,v_j],v_j]Zu, u\>| + o(h^3) \ltwo{u}^2. 
\end{align*}
Then use the self-adjointness of $Z$ and the H\"older inequality to write 
\begin{align*}
|\<Zh^{1-\tau} [[A,v_j],v_j]Zu, u\>|  &\leq \ltwo{h^{1-\tau} [[A,v_j], v_j] Z u} \ltwo{Zu}.
\end{align*}
Now note that for $k,l \geq 1,$ 
$$
h^{1-\tau} h^{(k+l) \tau} \Op(\bkl)=\Olt(h^{1+\tau}),
$$ 
and $h^{1-\tau}[[A,v_j], v_j]=\sum_{k, l \geq 1} h^{1-\tau} h^{(k+l) \tau} \Op(\bkl) + o(h^3) = \Olt(h^{1+\tau})$. Therefore 
\begin{align*}
\ltwo{h^{1-\tau} [[A,v_j], v_j] Z u} \ltwo{Zu}   & \leq C h^{1+\tau} \ltwo{Zu}^2.
\end{align*}
Now apply the elliptic estimate Lemma \ref{ellipticlemma} to $Zu$ to see 
\begin{align*}
C h^{1+\tau} \ltwo{Zu}^2  \leq C h^{6 \tau} \ltwo{f}^2 + C h^{1+\tau +2 } \ltwo{u}^2.
\end{align*}
Combining this chain of inequalities gives \eqref{anticommutatorest}.
\end{proof}

So now use \eqref{anticommutatorest} to estimate the sum of double commutators in \eqref{awop}
\begin{align*}
h^{3-\gamma-\tau} |\<(AW+WA)u, u\>| \leq Ch^{3-\tau} \lp{f}{2} \lp{u}{2} + Ch^{2+6\tau-\gamma} \lp{f}{2}^2 + o(h^3)\lp{u}{2}^2,
\end{align*}
which is exactly the desired inequality \eqref{awest}.

\subsection{Commutator Estimate of $A$ and $h^2 \p_x^2$}\label{hdxsubsec}
The Laplacian commutator estimate follows from writing $[h^2 \p_x^2, A]$ as a sum of a cutoff version of $P$, $J=\Op(\chi^{1/2}(x) \psi^{1/2}(\xi h^{\tau-1}))$  and error terms. The error terms are supported in the elliptic or damped set and can be further estimated. 

In particular in this subsection I will show
\begin{align}\label{kommutatorest}
2 h^3\beta \lp{J u}{2}^2 &\leq \left|\<[h^2 \p_x^2, A] u, u\>\right|+ C h^{3-\tau}\lp{f}{2} \lp{u}{2}+C h^{2+6\tau-\gamma} \lp{f}{2}^2 \\
&+ (C h^{3+2\tau}\beta+o(h^3))  \lp{u}{2}^2. \nonumber
\end{align}
To begin note by Lemma \ref{exoticcompose}  
\begin{align*}
[h^2 \p_x^2, h^{1-\tau} A] &= \sum_{k=0}^{2} \frac{(ih)^k}{2^k k!} (1-(-1)^{k}) \Op( \p_x^k(x \chi(x)) \xi \psi(\xi h^{\tau-1}) \p_{\xi}^k \xi^2) \\
&= 2 i h \Op((x \chi' + \chi) \xi^2 \psi(\xi h^{1-\tau})),
\end{align*}
where there are no terms for $k \geq 3$ since $\p_{\xi}^k \xi^2 =0$ for $k \geq 3$, and the $k=0,2$ terms cancel because $1-(-1)^k=0$ then.

Now recall that $p=\xi^2 + i h^{2- \gamma} W - h^2 \beta$ so 
\begin{equation}\label{eqlowerbound}
2ih ( x \chi' + \chi) \xi^2 \psi(\xi h^{\tau-1}) =  2ih x \chi' \xi^2 \psi +  2ih \chi \psi(\xi h^{\tau-1}) (p - i h^{2-\gamma} W + h^2 \beta).
\end{equation}
Each of the terms on the right hand side will be estimated in turn. The $h^3 \beta$ term will produce the $h^3 J$, the remaining terms produce errors. 

\subsubsection{Estimate of $\Op(\xi^2 \psi x \chi')$}\label{awkwardmaincommutatorsubsection}
To estimate $\Op(\xi^2 \psi x \chi')$ it is enough to use that $\chi'$ is supported only inside the damped set and apply \eqref{hdampest}. In order to do so the $\xi$ dependency of the operator must be eliminated, to do so the approach of Lemma \ref{ellipticestawaw} is adapted. Lemma \ref{commutebump} is still used, but the localizing function now depends on $x,$ and $\xi^2 \psi x \chi'$ is smooth in $x$ so the error term is $O(h^{\infty})$. 

In particular, define $s \in \Ci(-\pi, \pi)$ with 
$$
s=\begin{cases}
0 \quad |x| < \sigma \\
1 \quad \sigma+\sigma_1/2 < |x|,
\end{cases}
$$
and let $S$ be the operator of multiplication by $s$. Note $\supp(\chi') \subset \{\sigma+\sigma_1/2<|x|\}$ so $s \equiv 1$ on $\supp(\xi^2 \psi x \chi')$. Then by Lemma \ref{commutebump} (since $\chi' \in \Ci$)
\begin{align*}
\left|\<\Op(\xi^2 \psi(\xi h^{\tau-1}) x \chi'(x))u ,u\>\right| &= \left|\<S\Op(\xi^2 \psi(\xi h^{\tau-1}) x \chi'(x))Su ,u\>\right| +O(h^{\infty}) \lp{u}{2}^2. 
\end{align*}
Use that $S$ is self adjoint and the H\"older inequality to write 
\begin{align*}
\left|\<S\Op(\xi^2 \psi(\xi h^{\tau-1}) x \chi'(x))Su ,u\>\right| &=\left| \<\Op(\xi^2 \psi(\xi h^{\tau-1}) x \chi'(x))Su ,Su\>\right|\\
&\leq \lp{\Op(\xi^2 \psi(\xi h^{\tau-1}) x \chi'(x)) Su}{2} \lp{Su}{2}.
\end{align*}
By Lemma \ref{calderonvaillancourt}, $h^{2\tau-2}\Op(\xi^2 \psi( \xi h^{\tau-1}) x \chi'(x))$ is bounded on $L^2$, so
\begin{align*}
\lp{\Op(\xi^2 \psi(\xi h^{\tau-1}) x \chi'(x)) Su}{2} \lp{Su}{2}&\leq C h^{2-2\tau} \lp{Su}{2}^2.
\end{align*}
Then since $s \leq C W^{1/2}$ and applying \eqref{hdampest}
\begin{align*}
h^{2- 2 \tau} \lp{Su}{2}^2 &\leq C h^{2-2\tau} \lp{W^{1/2}u}{2}^2 \leq C h^{2-2\tau+\gamma} \lpfu.
\end{align*}
Combining this chain of inequalities and multiplying both sides by $h$ gives 
\begin{equation}\label{eqlowerbound2}
\awkwardtermmaincommutatorlemmaeq
\end{equation}

\subsubsection{Estimate of $\Op(h^2 \beta \chi \psi)$}
To estimate $\Op(\chi \psi)$ write it as $J^2$ plus an error term. 
By Lemma \ref{exoticcompose}
\begin{align*}
\Op(J) \Op(J) = \Op(\chi^{1/2} \psi^{1/2}) \Op(\chi^{1/2} \psi^{1/2}) &=\Op(\chi \psi) - h^{2\tau} \Op(r_1),
\end{align*}
 where $r_1 \in S^0_{1-\tau}$.
Using that $J$ is self adjoint and $\Op(r_1)$ is bounded on $L^2$ by Lemma \ref{calderonvaillancourt}
\begin{align*}
\lp{Ju}{2}^2= \left|\<\Op(\chi^{1/2} \psi^{1/2}) \Op(\chi^{1/2} \psi^{1/2}) u, u\>\right| &\leq \left|\< \Op(\chi \psi)u, u\>\right| + h^{2\tau} \<\Op(r_1) u, u\> \\
& \leq \left|\< \Op(\chi \psi)u, u\>\right| +C h^{2\tau} \lp{u}{2}^2.
\end{align*}
Therefore, multiplying through by $h^3 \beta$
\begin{equation}\label{eqlowerbound3}
\boundbelowJeq
\end{equation}

\subsubsection{Estimate of h $\Op(\chi \psi p)$}\label{chipsipsubsection}
To estimate  $\Op(\chi \psi p)$, write it as $\Op(\chi \psi)P$ plus error terms. The error terms are supported on the elliptic set of $P$ or the damped region and are further estimated using Lemma \ref{ellipticlemma} or \eqref{hdampest} respectively. In particular the following inequality will be shown.
\begin{equation}\label{eqlowerbound4}
\chipsiPintofeq
\end{equation}
Note this term appears in \eqref{eqlowerbound} as $h \Op(\chi \psi) p$, but to simplify notation this extra factor of $h$ is not carried through the intermediate calculations. Because of this remainders of size $o(h^2)$ are acceptably small, instead of the $o(h^3)$ of other calculations. 

To begin, $\Op(\chi \psi) P$ is computed, where special care must be taken with the regularity of the $W$ terms. 
Since $W \in W^{k_0, \infty}$ by part 2 of Lemmas \ref{lowregcompose} and \ref{exoticcompose} (Lemma \ref{exoticcompose} is used to compose $\Op(\chi \psi)$ and $-h^2 \p_x^2$, as Lemma \ref{lowregcompose} requires symbols to be bounded). 
\begin{align*}
 \Op(\chi \psi)P &= \Op(\chi \psi) (-h^2 \p_x^2 + i h^{2-\gamma} W  - h^2 \beta) \\
&= \Op(\chi \psi p) + \sum_{k=1}^{k_0-6} \frac{(ih)^k}{2^k k!} \Op\left( (\p_y \p_{\xi} - \p_{x} \p_{\eta})^k \chi(x) \psi(\xi h^{1-\tau}) (\eta^2 + i h^{2-\gamma} W(y)) \bigg|_{y=x, \eta=\xi} \right) \\
&+ \Olt(h^{(k_0-5)\tau -  5 (1-\tau)}).
\end{align*}
As in Lemma \ref{doublecommutator} since $\tau \in [\tm,1]$ with $\tm> \frac{7}{k_0-1}$
$$
(k_0-5) \tau - 5(1-\tau) = k_0 \tau  - 5 > \frac{7}{k_0-1} k_0 -5 > 2,
$$
which guarantees that the remainder term is of size $o(h^2)$. 

The sum splits into two separate sums, where the first is only taken to $k=2$ because $\p_{\xi}^k \xi^2 =0$ for $k \geq 3$.
\begin{align}\label{chipsimaineq}
\Op(\chi \psi) P &=\Op(\chi \psi p) + \sum_{k=1}^{2} \frac{(ih)^k}{2^k k!} (-1)^k \Op\left(\p_x^k \chi(x) \psi(\xi h^{\tau-1}) \p_{\xi}^k \xi^2 \right) \nonumber \\
&+ i h^{2-\gamma} \sum_{k=1}^{k_0-6} \frac{(ih)^k}{2^k k!} \Op(\chi(x) \p_{\xi}^k \psi(\xi h^{1-\tau}) \p_x^k W)+ \olt(h^2)\nonumber \\
&= \Op(\chi \psi p) - \Op(i h\chi' \psi \xi+\frac{h^2}{4} \chi'' \psi) + i h^{2-\gamma} \sum_{k=1}^{k_0-6} \frac{i^k}{2^k k!} h^{\tau k}\Op(\chi \psi^{(k)} \p_x^k W) + \olt(h^{2}).
\end{align}
The two operators and the sum will each be estimated individually. The $\Op(\chi \psi)P$ term is straightforward to control. The second term is supported inside the damped set and is controlled as in subsection \ref{awkwardmaincommutatorsubsection}. The sum will be controlled by the elliptic estimate using the same argument as Lemma \ref{ellipticestawaw}.

To begin, using the boundedness of $\Op(\chi \psi)$ on $L^2$ and that $Pu=h^2 f$ 
\begin{equation}\label{chipsieasy}
|\<\Op(\chi \psi) P u, u\>| \leq C h^2 \lp{f}{2} \lp{u}{2}.
\end{equation}

For the second term, set $g(x,\xi)= i h \chi' \psi \xi + \frac{h^2}{4} \chi'' \psi$. Note that $\chi'$ and $\chi''$ are supported inside the damping set and so an argument as in subsection \ref{awkwardmaincommutatorsubsection} will give an improvement. Recall $s \in \Ci(-\pi, \pi)$
$$
s(x) = \begin{cases} 
0 \quad |x| < \sigma\\
1 \quad \sigma+\sigma_1/2 <|x|,
\end{cases}
$$
and $S$ is the operator of multiplication by $s$. Since $g \in \Ci$ and $s \equiv 1$ on $\supp g$, by Lemma \ref{commutebump}  
\begin{align*}
\left|\<\Op(g) u,u \>\right| &= \left|\<S \Op(g) S u,u\>\right| + O(h^{\infty}) \lp{u}{2}^2
\end{align*}
Using that $S$ is self adjoint, along with the H\"older inequality
$$
\left|\<S \Op(g) S u,u\>\right|  \leq \ltwo{\Op(g) Su} \ltwo{Su}.
$$
Now by Lemma \ref{calderonvaillancourt}, $h^{\tau-2} \Op(g)$ is bounded on $L^2$, so
$$
\ltwo{\Op(g) Su} \ltwo{Su} \leq h^{2-\tau} \ltwo{Su}^2. 
$$
Then since $s \leq C W^{1/2}$ and applying \eqref{hdampest}
$$
h^{2-\tau} \ltwo{Su}^2 \leq C h^{2-\tau} \ltwo{W^{1/2} u}^2 \leq  C h^{2-\tau + \gamma} \lpfu. 
$$
Combining this chain of inequalities gives 
\begin{equation}\label{btwoest}
\left|\<\Op(i h\chi' \psi \xi + \frac{h^2}{4} \chi'' \psi) u,u \>\right| \leq C h^{2-\tau + \gamma} \lpfu + O(h^{\infty}) \lp{u}{2}^2.
\end{equation}

Now to estimate the sum, note that $\chi \psi^{(k)} \p_x^k W$ is supported in $\{2 h^{1-\tau} < |\xi| < 3 h^{1-\tau}  \}$ which is contained in the elliptic set. The proof of Lemma \ref{ellipticestawaw} will be imitated. Conjugate the $\chi \psi^{(k)} \p_x^k W$ terms in the sum in \eqref{chipsimaineq} by $Z$  to take advantage of the location of their support. Once again care is taken with the regularity of $\p_k W$ when applying Lemma \ref{commutebump}. 

Set $\wtbk(x,\xi) = \chi(x) \psi^{(k)}(\xi h^{1-\tau}) \p_x^k W(x)$. Recall $Z$ from Lemma \ref{ellipticlemma}. Since $z \equiv 1$ on $\supp(\wtbk)$ and $\p_x^k W \in W^{k_0-k, \infty},$ Lemma \ref{commutebump} can be applied with $N=k_0-k$ 
$$
\Op(\wtbk) = Z \Op(\wtbk) Z + \Olt(h^{(k_0-k)\tau -5}).
$$
So conjugating $\Op(\wtbk)$ by $Z$ is $\Op(\wtbk)$ modulo an error term. Once again terms with larger $k$ have less regularity and have larger error terms. However, as before, reintroducing the $h^{\tau k}$ from the sum improves the error terms 
$$
h^{\tau k} \Op(\wtbk) =  h^{\tau k} Z \Op(\wtbk)Z + h^{\tau k} \Olt(h^{(k_0-k)\tau -5}).
$$
In particular, and as in Lemma \ref{ellipticestawaw}, the error term is $o(h^2)$ because $\tau \geq \tm > \frac{7}{k_0-1}$ and
$$
\tau k + (k_0 - k) \tau - 5 = k_0 \tau -5  > \left(\frac{7}{k_0-1}\right)k_0 -5 >2.
$$
So 
$$
\left|\<h^{\tau k} \Op(\wtbk) u, u \>\right| \leq h^{\tau k}  \left|\<Z \Op(\wtbk) Z u, u\>\right| + o(h^2) \ltwo{u}^2.
$$
Continuing to follow the proof of Lemma \ref{ellipticestawaw}, use the self adjointness of $Z$ and the H\"older inequality to write 
$$
h^{\tau k}  \left|\<Z \Op(\wtbk) Zu, u\>\right| \leq  h^{\tau k} \ltwo{ \Op(\wtbk) Zu } \ltwo{Zu}.
$$
Now by Lemma \ref{calderonvaillancourt}, $\Op(\wtbk)$ is bounded on $L^2$
$$
h^{\tau k} \ltwo{ \Op(\wtbk) Zu } \ltwo{Zu} \leq h^{\tau k} \ltwo{Zu}^2. 
$$
Then apply the elliptic estimate, Lemma \ref{ellipticlemma}, to $Zu$
$$
h^{\tau k} \ltwo{Zu}^2 \leq C h^{5 \tau + k \tau-1} \ltwo{f}^2 + o(h^2) \ltwo{u}^2. 
$$
Combining this chain of inequalities and multiplying both sides by $h^{2-\gamma}$ gives 
\begin{equation}\label{boneest}
h^{2-\gamma} \left|\<h^{\tau k} \Op(\wtbk) u,u\>\right|\leq C h^{1-\gamma+(5+k)\tau} \lp{f}{2}^2  + o(h^{2}) \lp{u}{2}^2 \leq C h^{1-\gamma+6 \tau} \lp{f}{2}^2  + o(h^{2}) \lp{u}{2}^2.
\end{equation}
Where the second inequality follows since $k \geq 1$. 

Therefore using \eqref{chipsieasy}, \eqref{btwoest} and  \eqref{boneest} to estimate terms in \eqref{chipsimaineq}
$$
|\<\Op(\chi \psi p) u ,u\>| \leq C h^2 \lp{f}{2} \lp{u}{2} +C h^{2-\tau + \gamma} \lpfu +  Ch^{1-\gamma+6\tau} \lp{f}{2}^2  + o(h^{2}) \lp{u}{2}^2.
$$
Multiplying both sides by $h$ and using that $\gamma -\tau \geq0$ (since $\tau \in (1/2,1]$ and $\gamma \in \{1,2\}$) gives the desired inequality
\begin{equation*}
\chipsiPintofeq
\end{equation*}

\subsubsection{Estimate of $\Op(\chi\psi W)$}
To estimate $\Op(\chi\psi W)$ I write it as $v_j \Op(\chi \psi) v_j$ plus error terms. The $v_j \Op(\chi \psi) v_j$ terms are controlled by the damped region estimate \eqref{hdampest}. The error terms are either small or are supported on the elliptic set of $P$ and can be further estimated using Lemma \ref{ellipticlemma}. In particular the following inequality will be shown 
\begin{equation}\label{eqlowerbound5}
\dampingmaincommutatorlemmaeq
\end{equation}
Note this term appears in \eqref{eqlowerbound} as $h^{2-\gamma} h \Op(\chi \psi W)$, but to simplify notation these extra factors of $h$ are not carried through the intermediate calculations. Because of this remainders of size $o(h^2)$ are acceptably small, instead of the $o(h^3)$ of other calculations. 

To begin recall that $W = \sum v_j^2$ and so 
$$
\Op(\chi \psi W) = \Op\left(\chi \psi \sum v_j^2\right) = \sum \Op(\chi \psi v_j^2).
$$
This is exactly the principal symbol of $\sum v_j \Op(\chi \psi) v_j$, an expansion of which will now be computed.

First, since $v_j \in W^{k_0, \infty}$ apply part 2 of Lemma \ref{lowregcompose} with $\ti{N}=k_0-5$, to obtain
$$
\Op(\chi \psi) v_j = \sum_{k=0}^{k_0-6} \frac{(ih)^k}{2^k k!} \Op(\chi \p_{\xi}^k \psi(\xi h^{\tau-1}) \p_x^k v_j ) + \Olt(h^{(k_0-5) \tau - 5(1-\tau)}). 
$$
As in subsection \ref{chipsipsubsection}, since $\tau \geq \tm >\frac{7}{k_0-1}$ the remainder term is $o(h^2)$. In particular 
$$
(k_0-5) \tau - 5(1-\tau) = k_0 \tau - 5  > \left( \frac{7}{k_0-1} \right) k_0 -5 > 2. 
$$
Replacing $h^k \p_{\xi}^k \psi(\xi h^{\tau-1}) = h^{k\tau} \psi^{(k)}(\xi h^{\tau-1})$ gives
$$
\Op(\chi \psi) v_j = \sum_{k=0}^{k_0-6} \frac{i^k h^{\tau k}}{2^k k!} \Op(\chi \psi^{(k)} \p_x^k v_j ) + \olt(h^2). 
$$
Now compute the following composition of $v_j$ and $\Op(\chi \psi^{(k)} \p_x^k v_j)$, adjusting the number of terms taken in the expansion based on how many derivatives fallen on $\p_x^k v_j$. 
\begin{equation}\label{vjcomposum}
v_j \Op(\chi \psi) v_j = v_j \sum_{k=0}^{k_0-6} \frac{i^k}{2^k k!} h^{\tau k} \Op(\chi \psi^{(k)} \p_x^k v_j) + \olt(h^2).
\end{equation} 

In particular, since $v_j \in W^{k_0, \infty}$ and $\chi \p_x^k v_j \psi^{(k)} \in W^{k_0-k} S_{1-\tau}$, apply part 1 of Lemma \ref{lowregcompose} with $\ti{N}=k_0-5-k$ and $k_0-k \geq 5$. 
$$
\frac{i^k}{2^k k!} v_j h^{\tau k} \Op(\p_x^k v_j \psi^{(k)} \chi)=\sum_{l=0}^{k_0-k-6} \frac{(ih)^l}{2^l l!} \frac{i^k}{2^k k!} h^{\tau k} (-1)^l \Op(\chi \p_x^l v_j \p_x^k v_j \p_{\xi}^l \psi^{(k)}) + h^{\tau k} \Olt(h^{(k_0-k-5) \tau - 5(1-\tau)}).
$$

Although there are fewer terms taken in the expansion for larger values of $k$, the additional $h^{\tau k}$ ensures that the remainder term is $o(h^2)$. In particular  
$$
\tau k + (k_0-k-5)\tau - 5(1-\tau) =  k_0 \tau -5 >k_0  \left(\frac{7}{k_0-1}\right) -5 > 2.
$$
Therefore
$$
\frac{i^k}{2^k k!} v_j h^{\tau k} \Op(\p_x^k v_j \psi^{(k)} \chi)=\sum_{l=0}^{k_0-k-6} \frac{ h^{\tau(l+k)}}{2^{l+k} l! k!} i^{k+l} (-1)^l \Op(\p_x^l v_j \p_x^k v_j \psi^{(k+l)} \chi) + \olt(h^2).
$$
Now plug this into \eqref{vjcomposum} to obtain
$$
v_j \Op(\chi \psi) v_j =  \sum_{k=0}^{k_0-6} \left(\sum_{l=0}^{k_0-k-6} \frac{i^{(k+l)}}{2^{l+k} l! k!} h^{\tau(l+k)} (-1)^l \Op(\p_x^l v_j \p_x^k v_j \psi^{(k+l)} \chi) \right)+ \olt(h^2).
$$
Note that the $k=0, l=1$ term and $k=1, l=0$ term are identical except for a minus sign and cancel. There are more cancellations which occur in the sum, but only this first one is necessary for the proof. Note also that $k+l \leq k+ k_0-k-6 = k_0-6$,
\begin{equation}\label{vjchipsivjcompose}
v_j \Op(\chi \psi) v_j =  \Op(v_j^2 \chi \psi) + \sum_{k,l=1}^{k+l \leq k_0-6} \frac{i^{(k+l)}}{2^{l+k} l! k!} h^{\tau(l+k)} (-1)^l \Op(\p_x^l v_j \p_x^k v_j \psi^{(k+l)} \chi) + \olt(h^2).
\end{equation}
In order to further control the size of the terms in this sum the technique from Lemma \ref{ellipticestawaw} is used. Let $\bklt=\p_x^k v_j \p_x^l v_j \psi^{(k+l)} \chi$. For $k, l \geq 1, \bklt$ has support contained in the elliptic set which can be made use of byy conjugating by $Z$ as in Lemma \ref{ellipticestawaw} and then applying the elliptic estimate to $Zu$. The proof is almost identical to Lemma \ref{ellipticestawaw}, but is written here for exactness. 

As in Lemma \ref{ellipticestawaw}, recall $Z$ from Lemma \ref{ellipticlemma}. Note $z \equiv 1$ on $\supp \bklt$ and $\bklt \in W^{k_0-\max(k,l)}S_{1-\tau}(\Ps)$, so Lemma \ref{commutebump} with $N=k_0-\max(k,l)$ gives
\begin{align*}
\Op(\bklt) &= Z \Op(\bklt) Z +  \Olt(h^{\tau (k_0-\max(k,l))-5}).
\end{align*}
Note that there is less regularity for larger $l,k$ and so the error term is larger. However after reintroducing the $h^{(k+l) \tau}$ from the sum the error terms are improved
\begin{align*}
h^{(k+l)\tau} \Op(\bkl) &= h^{(k+l)\tau}  Z \Op(\bkl) Z + \Olt(h^{(k_0-\max(k,l))\tau-5} h^{\tau(k+l)}).\nonumber
\end{align*}
In particular since $\tau \geq \tm > \frac{7}{k_0-1}$ and $k, l \geq 1$ 
$$
(k_0-\max(k,l))\tau - 5 + \tau(k+l)  = (k_0 + \min(k,l)) \tau -5 > k_0 \left(\frac{7}{k_0-1} \right) -5 > 2.
$$
Therefore 
\begin{align}\label{ellipticvj}
h^{(k+l)\tau} \Op(\bkl) &=h^{(k+l)\tau}  Z \Op(\bkl) Z+\olt(h^{2}).
\end{align}
Now, apply \eqref{ellipticvj} term by term to $\tib=\sum_{k,l=1}^{k+l \leq k_0-6} C_{k,l} h^{\tau(l+k)} \bklt$, the sum in \eqref{vjchipsivjcompose}
\begin{align*}
|\<\Op(\ti{b}) u, u\>| &\leq |\<Z \Op(\ti{b}) Zu, u\>| + o(h^2) \ltwo{u}^2. 
\end{align*}
Then use the self-adjointness of $Z$ and the H\"older inequality to write 
\begin{align*}
|\<\Op(\ti{b}) u, u\>| \leq |\<Z \Op(\ti(b)) Zu, u\>|  &\leq \ltwo{\Op(\ti(b))Z u} \ltwo{Zu}.
\end{align*}
Now note, since $\bklt$ is bounded on $L^2$ by Lemma \ref{calderonvaillancourt}, and $k,l \geq 1,$ 
$$
h^{(k+l) \tau} \Op(\bklt)=\Olt(h^{2 \tau}),
$$ 
so $\Op(\tib)= \Olt(h^{2 \tau})$ and 
\begin{align*}
|\<\Op(\ti{b}) u, u\>| \leq \ltwo{\Op(\ti(b)) Z u} \ltwo{Zu}   & \leq C h^{2\tau } \ltwo{Zu}^2.
\end{align*}
Now apply the elliptic estimate Lemma \ref{ellipticlemma} to $Zu$ to see 
\begin{align}\label{tibestimate}
|\<\Op(\tib) u, u\>| \leq C h^{2 \tau } \ltwo{Zu}^2  \leq C h^{7 \tau-1} \ltwo{f}^2 + C h^{2+2 \tau } \ltwo{u}^2.
\end{align}
Now these pieces will be combined to give the final estimate of $\Op(\chi \psi W)$. Recall that $W = \sum_j v_j^2$ so 
$$
|\< \Op(\chi \psi W) u, u\>| \leq \sum_j |\<\Op(\chi \psi v_j^2)u, u\>|.
$$
The composition computation \eqref{vjchipsivjcompose} gives 
$$
\sum_j |\<\Op(\chi \psi v_j^2)u, u\>| \leq \sum_j |\<v_j \Op(\chi \psi) v_j u, u\>|  + |\<\Op(\ti{b})u,u\>|  +o(h^2) \lp{u}{2}^2.
$$
Then \eqref{tibestimate} gives 
\begin{equation}\label{chipsiwintermediateeq}
\sum_j |\<\Op(\chi \psi v_j^2)u, u\>| \leq \sum_j |\<v_j \Op(\chi \psi) v_j u, u\>|  + C h^{7 \tau-1} \ltwo{f}^2 +o(h^2) \lp{u}{2}^2.
\end{equation}
It remains to control the $v_j \Op(\chi \psi) v_j$ terms with the damping region estimate. Using that $v_j$ is a multiplier and thus is self-adjoint, as well as the H\"older inequality 
$$
|\<v_j \Op(\chi \psi) v_j u, u\>| \leq \ltwo{\Op(\chi \psi) v_j u} \ltwo{v_j u}.
$$
Now note $\Op(\chi \psi)$ is bounded on $L^2$ by Lemma \ref{calderonvaillancourt} so 
$$
\ltwo{\Op(\chi \psi) v_j u} \ltwo{v_j u} \leq C \ltwo{v_j u}^2. 
$$
Again using that $W= \sum v_j^2$ so $v_j^2 \leq W \leq C W^{1/2}$ and \eqref{hdampest}
\begin{align*}
 \ltwo{v_j u}^2 \leq C \ltwo{W^{1/2} u}^2 \leq C h^{\gamma}\lpfu. 
\end{align*}
Combining this chain of inequalities and \eqref{chipsiwintermediateeq} gives 
$$
|\<\Op(\chi \psi W)\>| \leq C h^{\gamma} \lpfu + h^{7 \tau-1} \ltwo{f}^2 + o(h^2) \ltwo{u}^2. 
$$
Finally multiply both sides by $h^{3-\gamma}$ to obtain the desired inequality
\begin{equation*}
\dampingmaincommutatorlemmaeq
\end{equation*}

\subsubsection{Combining Estimates}
Now use \eqref{eqlowerbound2}, \eqref{eqlowerbound3}, \eqref{eqlowerbound4} and \eqref{eqlowerbound5} to estimate terms in \eqref{eqlowerbound}  
\begin{align*}
2 h^3\beta \lp{J u}{2}^2 &\leq  h^{1-\tau} \left|\<[h^2 \p_x^2, A] u, u\>\right|
+ C (h^3 +h^{3-2\tau+\gamma}) \lp{f}{2} \lp{u}{2}\\
&+C h^{2+6\tau-\gamma} \lp{f}{2}^2 + (C h^{3+2\tau}\beta+o(h^3))  \lp{u}{2}^2. 
\end{align*}
Use that $\gamma -\tau >0$ to group the $\lpfu$ terms to obtain the desired estimate \eqref{kommutatorest}
\begin{align*}
2 h^3\beta \lp{J u}{2}^2 &\leq  h^{1-\tau} \left|\<[h^2 \p_x^2, A] u, u\>\right|+ C h^{3-\tau}\lp{f}{2} \lp{u}{2}+C h^{2+6\tau-\gamma} \lp{f}{2}^2 \\
&+ (C h^{3+2\tau}\beta+o(h^3))  \lp{u}{2}^2. \nonumber
\end{align*}

\subsection{End of Proof of Lemma \ref{mainestimate}}\label{synthesis}

Recall \eqref{startest} is
\begin{align*}
2 h^{3-\tau} Im \<f, Au\> &= h^{1-\tau} \<[h^2\p_x^2, A] u,u\> + h^{3-\gamma-\tau} \<(AW+WA)u,u\>.
\end{align*}
Now apply \eqref{awest}, \eqref{kommutatorest}, to estimate the terms on the right hand side, and Lemma \ref{calderonvaillancourt} (to see that $\lp{Au}{2} \leq C\lp{u}{2}$) 
\begin{align*}
2 h^3 \beta \lp{Ju}{2}^2 & \leq Ch^{3-\tau} \lp{f}{2}\lp{u}{2} + C h^{2+6\tau-\gamma} \lp{f}{2}^2 + C\left(o(h^3)+ \beta h^{3+2\tau}\right) \lp{u}{2}^2.
\end{align*}
Divide through by $2 h^3 \beta$ to obtain the desired estimate, which can be done since $\beta$ is bounded away from 0. 
\end{proof}

\appendix
\section{Pseudodifferential Operators}\label{pseudosection}
This appendix contains the necessary background information on pseudodifferential operators, as well as a lemma calculating the size of errors from introducing cutoff operators and a careful calculation of the regularity required to have remainder terms in composition expansions bounded on $L^2$. 

This paper uses the semiclassical Weyl quantization, which takes in a function on $T^* \Rb$ and produces an operator $\Op(a)$ defined by 
\begin{equation}
\Op(a) u(x) = \frac{1}{2\pi h} \intr \intr e^{\frac{i(x-y)\xi}{h}} a\left(\frac{x+y}{2}, \xi\right)u(y) dy d\xi.
\end{equation}
On the torus this formula still makes sense. A function $a \in \Ci(\Ps)$ is equivalent to $a \in \Ci(\Rb_x \times \Rb_{\xi})$ periodic in the $x$ variable. It is straightforward to see that for such $a, \Op(a)$ preserves the space of $2\pi \Zb$ periodic distributions on $\Rb$ and thus preserves $\Dc'(\mathbb{S}^1)$. 
\begin{definition} \label{hdepsymbol}
 $a(x,\xi; h) \in \Smr$  if $a \in \Ci(\Ps)$ and satisfies
\begin{equation}
\sup_{x,\xi} |\p_{x}^{\alpha} \p_{\xi}^{\theta} a(x, \xi; h)| \leq C_{\alpha \theta} h^{-\rho|\theta|} \<\xi\>^{m-|\theta|}.
\end{equation}
\end{definition}
Note that this definition is not the typical one for $h$ dependent symbols. In particular only derivatives in $\xi$ produce unfavorable powers of $h$, derivatives in $x$ do not produce any. This structure would allow $\rho \geq 1/2$ (see \cite{DyatlovZahl} section 3) corresponding to $\tau<1/2$ which would give an improved decay rate, however requirements of the elliptic estimate (Proposition \ref{ellipticlemma}) prevent $\rho$ from being taken this large. 

The following lemma gives the standard composition and adjoint formula for $\Smr$ symbols. It follows from Theorems 4.17 and 4.18 of \cite{Zworski2012}. 
\begin{lemma}\label{exoticcompose}
Let $a \in \Smr, b \in \Smrp$ then
\begin{enumerate}
	\item $\Op(a) \Op(b)=\Op(a\#b)$ where $a\# b \in \Smrpp$ and for each $N$ 
	\begin{align}
	a\#b(x,y; h) &= \sum_{k=0}^{N-1} \frac{(ih)^k}{2^k k!} \left( \p_y \p_{\xi} - \p_x \p_{\eta}\right)^k (a(x,\xi;h) b(y,\eta;h))\bigg|_{y=x, \eta=\xi} + O_{\Smrpp}(h^{(N(1-\rho)}).
	\end{align}
	\item $\Op(a)^* = \Op(\bar{a}),$ in particular real symbols have self-adjoint Weyl quantization.
\end{enumerate}
\end{lemma}

The following two definitions are finite regularity analogs of Definition \ref{hdepsymbol}. In particular they define two different symbol classes with a finite number of derivatives in $x$ and an infinite number of derivatives in $\xi$. The first only produces unfavorable powers of $h$ when differentiated in $\xi$ while the second produces unfavorable powers of $h$ when differentiated in $\xi$ and $x$. The notation is again somewhat unusual but is made this way to mirror Definition \ref{hdepsymbol}. 
\begin{definition}\label{lowregsymboldef}
A distribution $a \in W^k S_{\rho}(\Ps)$ if for $\alpha \leq k, \theta \in \N$ 
\begin{align*}
\sup_{x, \xi} |\p_x^{\alpha} \p_{\xi}^{\theta} a| \leq C h^{-\rho \theta} \<\xi\>^{-\theta}.
\end{align*}
A distribution $a \in W^k S_{\rho, \rho}(\Ps)$ if for $\alpha \leq k, \theta \in \N$
\begin{align*}
\sup_{x, \xi} |\p_x^{\alpha} \p_{\xi}^{\theta} a| \leq C h^{-\rho (\alpha+\theta)} \<\xi\>^{-\theta}.
\end{align*}
\end{definition}

The following theorem gives a sufficient condition for a pseudodifferential operator to be bounded on $L^2$. It follows immediately from Theorem 1.2 of \cite{Boulkhemair1999}.
\begin{lemma}\label{calderonvaillancourt}
There exists $C>0$ such that for all $b(x,\xi) \in \Sc'(\Ps)$ 
$$
\nm{\Op(b)}_{\Ls(L^2(\Sb^1))} \leq C \sum_{\alpha, \theta \in \{0,1\}} h^{\theta} \lp{\p_x^{\alpha} \p_{\xi}^{\theta} b}{\infty}.
$$
In particular if $b \in W^1 S_{\rho}(\Ps)$ then $\Op(b)$ is bounded on $L^2$.
\end{lemma}
\begin{proof}
So
\begin{align*}
\Op(b) &= (2\pi h)^{-1} \int_{\Rb \times \Rb} e^{\frac{i(x-y)\xi}{h}} b\left(\frac{x+y}{2}, \xi\right) v(y) dy d \xi \\
&= (2\pi)^{-1} \int_{\Rb \times \Rb} e^{i(x-y) \eta} b\left( \frac{x+y}{2}, \eta h\right) v(y) dy d \eta =\Op(b(\cdot, h \eta))
\end{align*}
which by \cite{Boulkhemair1999} Theorem 1.2 has  
$$
\nm{\Op(b)}_{\Ls(L^2(\Sb^1))} \leq C \sum_{\alpha, \theta \in \{0,1\}} \lp{\p_x^{\alpha} \p_{\eta}^{\theta} b(x,h \eta)}{\infty} \leq C \sum_{\alpha, \theta \in \{0,1\}} h^{\theta} \lp{\p_x^{\alpha} \p_{\xi}^{\theta} b}{\infty}.
$$
\end{proof}

In order to prove composition results for finite regularity symbols I will make use of the notation and results of \cite{SjostrandWiener}
\begin{definition}
Let $e_1, \ldots, e_m$ be a basis in $\Rn$ and $\Gamma=\bigoplus_1^n \Zb e_j$. Then let $\chi_0 \in \Sc(\Rn)$ be such that $1=\sum_{j\in \Gamma} \chi_j(x)$ where $\chi_j(x) =\chi_0(x-j)$ for $j \in \Gamma$. Define $S_w$ as the space of $u \in \Sc'(\Rn)$ such that 
$$
U(\xi) = \sup_{j \in \Gamma} |\F \chi_j u(\xi)| \in L^1(\Rn).
$$
Then $S_w$ is a Banach space with the norm
$$
\|u\|_{\Gamma,\chi_0} = \lp{\sup_{j \in \Gamma} |\F \chi_j u| }{1}.
$$
\end{definition}

The following $L^2$ boundedness result is from page 8 of \cite{SjostrandWiener} . 
\begin{lemma}\label{swl2bound}
If $a \in S_w$ then $\Op(a)$ is bounded on $L^2$ and 
$$
\| \Op(a) \|_{L^2 \ra L^2} \leq \Sw{a}. 
$$
\end{lemma}

If $k$ is taken large enough then $W^k S^m(\Rn)$ is contained in $S_w(\Rn)$. Note that this result is stated on a more general space than $\Ps$. This is because in the computation of an expansion of the composition of symbols $a,b$ there is an intermediate step where $c(x,\xi,y,\eta)=a(x,\xi) b(y,\eta)$ is considered as a symbol on $\Ps \times \Ps$, which can be thought of as $\Rb^4$.
\begin{lemma}\label{swtofinitelemma}
If $a \in W^kS^m_{\rho}(\Rn)$ for $k \geq n+1$ then $a \in S_w(\Rn)$ and 
$$
\Sw{a} \leq C \sup_{|\gamma| \leq n+1} \lp{ \p^{\gamma} a}{\infty}. 
$$
\end{lemma}
\begin{proof}
Starting with the definition of $\Sw{\cdot}$
\begin{align*}
\Sw{a} &= \int_{\Rn} \sup_{j \in \Gamma} |\F(\chi_j u)(\xi)| d\xi \\
&=\int_{\Rn} \<\xi\>^{n+1} \<\xi\>^{-(n+1)} \sup_{j \in \Gamma} |\F(\chi_j u)(\xi)| d\xi \\
&\leq C \lp{\<\xi\>^{n+1} \sup_{j \in \Gamma} |\F(\chi u)(\xi)| }{\infty} \leq \sup_{|\alpha| \leq n+1} \lp{\xi^{\alpha} \sup_{j \in \Gamma} |\F(\chi_j u)(\xi)|}{\infty}.
\end{align*}
Where the integrability of $\<\xi\>^{-(n+1)}$ on $\Rn$ gives the first inequality. Then 
\begin{align*}
\sup_{|\alpha| \leq n+1} \lp{\xi^{\alpha} \sup_{j \in \Gamma} |\F(\chi_j u)(\xi)|}{\infty} &\leq \sup_{|\alpha| \leq n+1} \lp{ \sup_{j \in \Gamma} | \xi^{\alpha} \F(\chi_j u)(\xi)|}{\infty} \\
&= \sup_{|\alpha| \leq n+1} \lp{ \sup_{j \in \Gamma} | \F( \p^{\alpha}(\chi_j u))(\xi)|}{\infty}. \\
\end{align*}
and 
$$
|\F(\p^{\alpha} (\chi_j u))(\xi)| = \left|\int e^{-ix\xi} \p^{\alpha}(\chi_j u) dx \right| \leq \int |\p^{\alpha} (\chi_j u)| dx \leq C \lp{\p^{\alpha} u}{\infty} \int \chi_j dx.
$$
\end{proof}

The following lemma gives an exact calculation of the regularity required to show the remainder term in a composition is bounded on $L^2$. The $S_w$ symbol class is used here as it allows a more straightforward proof than proceeding directly with $W^k S_{\rho}$ symbols. 
\begin{lemma}\label{dreaded}
If $a,b \in \Sc'(\Rb^{2n})$ with $(\p_y \p_{\xi} - \p_x \p_{\eta})^N a(x,\xi) b(y,\eta) \in S_w(\Rb^{4n})$ for some $N \in \Nb$ and $Q$ is a symmetric nonsingular matrix define 
$$
R_N(a,b)(x,\xi) = \int_0^1 (1-t)^{N-1} e^{it h \<QD,D\>} \left(\p_y \p_{\xi} - \p_x \p_{\eta}\right)^N (a(x,\xi; h) b(y, \eta; h)) dt \bigg|_{y=x, \eta=\xi}.
$$
then for $h$ chosen small enough
\begin{equation}\label{eq:dreaded1}
\Sw{R_N(a,b)} \leq C \Sw{ \left(\p_y \p_{\xi} - \p_x \p_{\eta}\right)^N (a(x,\xi) b(y,\eta))}.
\end{equation}
Therefore $\Op(R_N)$ is bounded as an operator on $L^2(\Rb^n)$ with 
$$
\nm{\Op(R_N)}_{L^2 \ra L^2} \leq \Sw{ \left(\p_y \p_{\xi} - \p_x \p_{\eta}\right)^N (a,b)} \leq \sup_{|\gamma|\leq4n+1} \lp{\p^{\gamma} \left(\p_y \p_{\xi} - \p_x \p_{\eta}\right)^N (a,b)}{\infty}.
$$
\end{lemma}
\begin{proof}
By \cite{SjostrandWiener} Theorem 1.4 and equation (1.21) (pg. 7), for any $\e>0$ there exists $h_0>0$ such that for $h<h_0$ 
$$
\Sw{\int_0^1 e^{ith\<QD,D\>} \left(\p_y \p_{\xi}-\p_x \p_{\eta}\right)^N (a,b) dt - C \left(\p_y \p_{\xi}-\p_x \p_{\eta}\right)^N (a,b)} \leq \e.
$$
Then \eqref{eq:dreaded1} follows by choosing $\e<\Sw{\left(\p_y \p_{\xi}-\p_x \p_{\eta}\right)^N (a,b)}$ and using the fact that restriction to a linear subspace (i.e. setting $y=x, \eta=\xi$) is bounded on $S_w$ (bottom of page 2 in \cite{SjostrandWiener}). The $L^2$ bound then follows by Lemmas \ref{swl2bound} and \ref{swtofinitelemma}, where $|\gamma|\leq4n+1$ because $a(x,\xi)b(y,\eta)$ is a function on $\Rb^{4n}$.
\end{proof}
The following lemma is a composition expansion result for low regularity symbols. The expansion is of the same form as Lemma \ref{exoticcompose}, but the remainder term is larger by a factor of $h^{-5\rho}$. 
\begin{lemma}\label{lowregcompose} Suppose $N \geq 6 \in \Nb$ is some fixed constant and $a, b$ are distributions  
\begin{enumerate} 
	\item  If $b \in W^{N,\infty}(\mathbb{S}^1)$, let $\ti{N} \leq N-5$ and assume for some $\rho \in [0,1/2)$ that $a \in W^{5} S_{\rho}(\Ps)$. Then 
\begin{align*}
\Op(a) \Op(b) &= \sum_{k=0}^{\ti{N}-1} \frac{(ih)^k}{2^k k!} \Op(\p_{\xi}^k a(x,\xi) \p_x^k b(x)) + \Olt(h^{\ti{N}(1-\rho)-5\rho}).\\
\Op(b) \Op(a) &= \sum_{k=0}^{\ti{N}-1} \frac{(ih)^k}{2^k k!}(-1)^k \Op(\p_{\xi}^k a(x,\xi) \p_x^k b(x)) + \Olt(h^{\ti{N}(1-\rho)-5\rho}).
\end{align*}

	\item If for some $\rho \in [0,1/2), b \in W^N S_{\rho}(\Ps)$ and $a \in S_{\rho}(\Ps)$ let $\ti{N} \leq N-5$ then 
	\begin{align*}
	\Op(a) \Op(b) &= \sum_{k=0}^{\ti{N}-1} \frac{(ih)^k}{2^k k!}  \Op\left((\p_y \p_{\xi} - \p_x \p_{\eta})^k a(x,\xi) b(y,\eta) \bigg|_{y=x, \eta=\xi} \right) + \Olt(h^{\ti{N}(1-\rho)-5\rho}).\\
	\Op(b) \Op(a) &= \sum_{k=0}^{\ti{N}-1} \frac{(ih)^k}{2^k k!}  \Op\left((\p_y \p_{\xi} - \p_x \p_{\eta})^k b(x,\xi) a(y,\eta) \bigg|_{y=x, \eta=\xi} \right) + \Olt(h^{\ti{N}(1-\rho)-5\rho}).\\
	\end{align*}
	
	\item If for some $\rho \in [0,1/2), b \in  W^N S_0(\Ps)$ and $a \in W^N S_{\rho, \rho}(\Ps)$ let $\ti{N} \leq N-5$ then 
	\begin{align*}
	\Op(a) \Op(b) &= \sum_{k=0}^{\ti{N}-1} \frac{(ih)^k}{2^k k!}  \Op\left((\p_y \p_{\xi} - \p_x \p_{\eta})^k a(x,\xi) b(y,\eta) \bigg|_{y=x, \eta=\xi} \right) + \Olt(h^{\ti{N}(1-\rho)-5\rho}).\\
	\end{align*}
\end{enumerate}
\end{lemma}
Cases 1 and 2 are stated separately to emphasize that when one symbol depends only on $x$, less regularity is required of the other symbol. Also note that in Case 3 one of the symbols does not produce unfavorable powers of $h$ under differentiation.

\begin{proof}
The proof relies on \cite{SjostrandWiener}, although special attention is paid to the minimal regularity necessary. Only the proof of the first part of 1) will be show, the other parts follow by analogous arguments. 

Set $c(x,\xi,y,\eta)=a(x,\xi)b(y,\eta)$ and let $Q$ be the symmetric nonsingular matrix given by $\<QD, D\>_{\Rb^4} = \<D_{\xi}, D_y\>_{\Rb^2} -\<D_{\eta}, D_x\>_{\Rb^2}$ where $D=(D_x, D_{\xi}, D_y, D_{\eta})$. I will first show that 
\begin{equation}\label{compose1}
\Op(a)\Op(b) = \Op((e^{\frac{ih}{2} \<QD,D\>} c)(x,\xi, x,\xi)),
\end{equation}
and then provide an expansion of the right hand side of the desired form. 

To begin since $a,b \in W^{k} S_{\rho}(\Ps)$ for $k \geq =2+1=3$,  then $a,b \in S_w(\Ps)$ by Lemma \ref{swtofinitelemma}. Therefore by \cite{SjostrandWiener} [Theorem 2.2, and the discussion on pages 7-8] \eqref{compose1} holds, where $e^{\frac{ih}{2} \<QD,D\>}$ is defined as the unique extension from $\Sc$.

So it remains to provide an expansion of $(e^{\frac{ih}{2}\<QD,D\>} c)(x,\xi, x,\xi)$. Well using a standard Taylor expansion of $e^{\frac{ih}{2}\<QD,D\>}$ as in \cite{SjostrandWiener} equation (1.20)

\begin{align*}
e^{\frac{ih}{2}\<QD,D\>} c(x,\xi,y,\eta) &= \sum_{k=0}^{\ti{N}-1} \frac{(ih)^k}{2^k k!} \<QD, D\>^k c(x,\xi,y,\eta)\\
&+\frac{(ih)^{\ti{N}}}{2^{\ti{N}}\ti{N}!} \int_0^1 (1-t)^{\ti{N}-1}e^{\frac{ih}{2}\<QD,D\>} \<QD, D\>^{\ti{N}} c(x,\xi,y,\eta) dt \\
&= \sum_{k=0}^{\ti{N}-1} \frac{(ih)^k}{2^k k!} (\p_y \p_{\xi} - \p_x \p_{\eta})^k a(x,\xi) b(y)\\
&+\frac{(ih)^{\ti{N}}}{2^{\ti{N}}\ti{N}!} \int_0^1 (1-t)^{\ti{N}-1}e^{\frac{ih}{2}\<QD,D\>} (\p_y \p_{\xi} - \p_x \p_{\eta})^{\ti{N}} a(x,\xi) b(y) dt \\
&= \sum_{k=0}^{\ti{N}-1} \frac{(ih)^k}{2^k k!} \p_{\xi}^k a(x,\xi) \p^k_y b(y)+\frac{(ih)^{\ti{N}}}{2^{\ti{N}}\ti{N}!} \int_0^1 (1-t)^{\ti{N}-1}e^{\frac{ih}{2}\<QD,D\>}  \p_{\xi}^{\ti{N}} a(x,\xi) \p_{y}^{\ti{N}} b(y) dt. 
\end{align*}
Therefore 
$$
e^{\frac{ih}{2}\<QD,D\>} c(x,\xi, x,\xi) = \sum_{k=0}^{\ti{N}-1} \frac{(ih)^k}{2^k k!} \p_{\xi}^k a(x,\xi) \p_x^k b(x) + \frac{(ih)^{\ti{N}}}{2^{\ti{N}} \ti{N}!} \int_0^1 (1-t)^{\ti{N}-1} e^{\frac{ih}{2} \<QD,D\>} \p_{\xi}^{\ti{N}}a(x,\xi) \p_y^{\ti{N}} b(y) dt |_{y=x, \eta=\xi}.
$$
First consider the terms in the sum. Since $a$ is smooth in $\xi$, $b \in W^{N,\infty}$ and $k \leq \ti{N}-1 \leq N-6$, for $\alpha \leq 1$ and $\theta \in \Nb$
$$
|\p_{\xi}^{k+\theta} a(x,\xi) \p_x^{k+\alpha} b(x)| < \infty.
$$
In particular each term in the sum is in $W^1 S_{\rho}(\Ps)$ and so quantizing it produces an operator bounded on $L^2$ by Lemma \ref{calderonvaillancourt}.

Now consider the integral term. By Lemma \ref{dreaded} it is in $S_w$ and its quantization is bounded on $L^2$ by 
\begin{align*} 
&C h^{\ti{N}} \sup_{x,y,\xi,\eta, |\gamma| \leq 5}  \left|\p^{\gamma} (\p_{y} \p_{\xi})^{\ti{N}} (a(x,\xi) b(y) ) \right| \\
&\leq C h^{\ti{N}} \sup_{x,y,\xi,\eta} \sum_{\gamma_1+\gamma_2+ \gamma_3 \leq 5}\left| \p_y^{\ti{N}+\gamma_1} b(y) \p_x^{\gamma_2} \p_{\xi}^{\ti{N}+\gamma_3} a(x,\xi) \right| \\
&\leq C h^{\ti{N}} h^{-\rho(\ti{N}+5)},
\end{align*}
where the final inequality holds because $a \in W^{5} S_{\rho}(\Ps)$ and $b \in W^{N, \infty}(\Sb^1)$ with $\ti{N} +5 \leq N$.
\end{proof}

The following lemma calculates the size of errors from introducing cutoff operators. It is a key tool used to take advantage of symbols with support contained in regions of phase space where good estimates hold, namely the elliptic set of $P$ and the support of $W$. 
\begin{lemma}\label{commutebump}
Fix $N \in \Nb, N \geq 6$ and $\rho \in [0,1/2)$. Suppose $ b \in W^{N}S_{\rho}(\Ps)$ and $t \in S^{0}_{\rho}(\Ps)$, such that $t \equiv 1$ on $\supp b$, then
\begin{enumerate}
\item
\begin{align*}
\Op(t) \Op(b) &= \Op(b) + \Olt(h^{N(1-\rho)-5})\\
 \Op(b) \Op(t) &=\Op(b) + \Olt(h^{N(1-\rho)-5}).
\end{align*}
\item $$
\Op(t) \Op(b) \Op(t) = \Op(b) + \Olt(h^{N(1-\rho)-5}).
$$ 
\end{enumerate}
\end{lemma}
\begin{proof}
Well by part 2 of Lemma \ref{lowregcompose}, setting $\ti{N}=N-5$
\begin{align*}
\Op(t) \Op(b) &= \sum_{k=0}^{N-6} \frac{(ih)^k}{2^k k!} \Op\left( (\p_y \p_{\xi} - \p_x \p_{\eta})^k t(x,\xi) b(y,\eta)\bigg|_{y=x,\eta=\xi} \right) + \Olt(h^{(N-5)(1-\rho) - 5\rho}) \\
&= \Op(tb) + \Olt(h^{N(1-\rho)-5}) \\
&= \Op(b) + \Olt(h^{(N(1-\rho)-5}),
\end{align*}
where the terms with $1\leq k \leq N-6$ all vanish, since $\p_{\xi}^k t(x,\xi) = \p_x^k t(x,\xi)=0$ on $\supp b$, and $\Op(tb)=\Op(b)$ as $t \equiv 1$ on $\supp b$. The second equation of part 1 follows by an analogous proof.

To see part 2 use the first half of part 1 of this Lemma  
\begin{align}
\Op(t) \Op(b) \Op(t) &= \left(\Op(t) \Op(b) \right) \Op(t) \nonumber \\
&= \left( \Op(b) + \Olt(h^{N(1-\rho)-5}) \right) \Op(t) \nonumber \\
&= \Op(b) \Op(t) + \Olt(h^{N(1-\rho)-5}) \label{finalequation}
\end{align}
where $\Op(t)$ is bounded on $L^2$ by Lemma \ref{calderonvaillancourt}. Now apply the second half of part 1 of this Lemma to $\Op(b) \Op(t)$ to obtain the desired conclusion. 
\end{proof}

\bibliographystyle{alpha}
\bibliography{mybib}

\begin{thebibliography}{CSVW14}

\bibitem[AL14]{AnantharamanLeautaud2014}
N.~Anantharaman and M.~L{\'e}autaud.
\newblock Sharp polynomial decay rates for the damped wave equation on the
  torus.
\newblock {\em Anal. PDE}, 7(1):159--214, 2014.
\newblock With an appendix by S. Nonnenmacher.

\bibitem[BG97]{BurqGerard1997}
N.~Burq and P.~G\'erard.
\newblock Condition n\'ecessaire et suffisante pour la contr\^olabilit\'e
  exacte des ondes.
\newblock {\em , C. R. Math. Acad. Sci. Paris}, 325(7):749--752, 1997.

\bibitem[BG18]{BurqGerard2018}
N.~Burq and P.~G{\'e}rard.
\newblock Stabilisation of wave equations on the torus with rough dampings.
\newblock {\em arXiv:1801.00983}, 2018.

\bibitem[BH07]{BurqHitrik2007}
N.~Burq and M.~Hitrik.
\newblock Energy decay for damped wave equations on partially rectangular
  domains.
\newblock {\em Mathematical Research Letters}, 14(1):35--47, 2007.

\bibitem[BLR92]{BardosLebeauRauch1992}
C.~Bardos, G.~Lebeau, and J.~Rauch.
\newblock Sharp sufficient conditions for the observation, control, and
  stabilization of waves from the boundary.
\newblock {\em SIAM Journal on Control and Optimization}, 30(5):1024--1065,
  1992.

\bibitem[Bon05]{Bony2005}
J.M. Bony.
\newblock Sommes de carr{\'e}s de fonctions d{\'e}rivables.
\newblock {\em Bulletin de la soci{\'e}t{\'e} math{\'e}matique de France},
  133(4):619--639, 2005.

\bibitem[Bou99]{Boulkhemair1999}
A.~Boulkhemair.
\newblock L2 estimates for weyl quantization.
\newblock {\em Journal of Functional Analysis}, 165(1):173 -- 204, 1999.

\bibitem[BT10]{BorichevTomilov2010}
A.~Borichev and Y.~Tomilov.
\newblock Optimal polynomial decay of functions and operator semigroups.
\newblock {\em Mathematische Annalen}, 347(2):455--478, 2010.

\bibitem[Bur98]{Burq1998}
N.~Burq.
\newblock D\'ecroissance de l'\'energie locale de l'\'equation des ondes pour
  le probl\`eme ext\'erieur et absence de r\'esonance au voisinage du r\'eel.
\newblock {\em Acta Math.}, 180(1):1--29, 1998.

\bibitem[BZ12]{BurqZworski2012}
N.~Burq and M.~Zworski.
\newblock {Control for Schr\"odinger operators on tori}.
\newblock {\em Mathematical Research Letters}, 19(2):309--324, 2012.

\bibitem[BZ19]{BurqZworski2019}
N.~Burq and M.~Zworski.
\newblock {Rough controls for Schr\"odinger operators on tori}.
\newblock {\em Annales Henri Lebesgue}, 2:331--347, 2019.

\bibitem[CSVW14]{csvw}
H.~Christianson, E.~Schenck, A.~Vasy, and J.~Wunsch.
\newblock From resolvent estimates to damped waves.
\newblock {\em J. Anal. Math.}, 121(1):143--162, 2014.

\bibitem[DJN19]{DyatlovJinNonnenmacher}
S.~Dyatlov, L.~Jin, and S.~Nonnenmacher.
\newblock Control of eigenfunctions on surfaces of variable curvature.
\newblock {\em arXiv:1906.08923}, 2019.

\bibitem[DK19]{DatchevKleinhenz2019}
K.~Datchev and P.~Kleinhenz.
\newblock {Sharp polynomial decay rates for the damped wave equation with
  H\"older-like damping}.
\newblock {\em arXiv:1908.05631}, 2019.

\bibitem[DZ16]{DyatlovZahl}
S.~Dyatlov and J.~Zahl.
\newblock Spectral gaps, additive energy, and a fractal uncertainty principle.
\newblock {\em Geometric and Functional Analysis}, 26(4):1011--1094, 2016.

\bibitem[DZ19]{DyatlovZworski2020}
S.~Dyatlov and M.~Zworski.
\newblock {\em Mathematical theory of scattering resonances}, volume 200.
\newblock American Mathematical Soc., 2019.

\bibitem[{Jaf}90]{Jaffard1990}
S.~{Jaffard}.
\newblock {Contr\^ole interne exact des vibrations d'une plaque rectangulaire.
  (Internal exact control for the vibrations of a rectangular plate).}
\newblock {\em {Port. Math.}}, 47(4):423--429, 1990.

\bibitem[Joh02]{JohnsondiBruno}
W.P. Johnson.
\newblock The curious history of fa{\`a} di bruno's formula.
\newblock {\em The American mathematical monthly}, 109(3):217--234, 2002.

\bibitem[Kle19]{Kleinhenz2019}
P.~Kleinhenz.
\newblock {Stabilization Rates for the Damped Wave Equation with
  H\"older-Regular Damping}.
\newblock {\em Commun. Math. Phys.}, 369(3):1187--1205, 2019.

\bibitem[Leb96]{Lebeau1996}
G.~Lebeau.
\newblock Equation des ondes amorties.
\newblock In {\em Algebraic and Geometric Methods in Mathematical Physics:
  Proceedings of the Kaciveli Summer School, Crimea, Ukraine, 1993}, pages
  73--109. Springer Netherlands, Dordrecht, 1996.

\bibitem[LR05]{LiuRao2005}
Z.~Liu and B.~Rao.
\newblock Characterization of polynomial decay rate for the solution of linear
  evolution equation.
\newblock {\em Zeitschrift f{\"u}r angewandte Mathematik und Physik ZAMP},
  56(4):630--644, 2005.

\bibitem[{Mac}10]{Macia2010}
F.~{Maci\`a}.
\newblock {High-frequency propagation for the Schr\"odinger equation on the
  torus.}
\newblock {\em {J. Funct. Anal.}}, 258(3):933--955, 2010.

\bibitem[Ral69]{Ralston1969}
J.~Ralston.
\newblock Solutions of the wave equation with localized energy.
\newblock {\em Communications on Pure and Applied Mathematics}, 22(6):807--823,
  1969.

\bibitem[RT75]{RauchTaylor1975}
J.~Rauch and M.~Taylor.
\newblock Exponential decay of solutions to hyperbolic equations in bounded
  domains.
\newblock {\em Indiana Univ. Math. J.}, 24(1):79--86, 1975.

\bibitem[Sj{\"o}95]{SjostrandWiener}
J.~Sj{\"o}strand.
\newblock Wiener type algebras of pseudodifferential operators.
\newblock {\em S{\'e}minaire {\'E}quations aux d{\'e}riv{\'e}es partielles
  (Polytechnique) dit aussi "S{\'e}minaire Goulaouic-Schwartz"}, 1994-1995.
\newblock talk:4.

\bibitem[Sta17]{Stahn2017}
R.~Stahn.
\newblock Optimal decay rate for the wave equation on a square with constant
  damping on a strip.
\newblock {\em Zeitschrift f{\"u}r angewandte Mathematik und Physik}, 68(2):36,
  2017.

\bibitem[Zwo12]{Zworski2012}
M.~Zworski.
\newblock {\em Semiclassical analysis}, volume 138.
\newblock American Mathematical Soc., 2012.

\end{thebibliography}
\end{document}